\documentclass{amsart}
\usepackage{amssymb,amsmath}
\usepackage{graphicx}
\usepackage{subfigure}
\usepackage{ upgreek }
\usepackage[all]{xy}
\usepackage{mathtools}
\usepackage{tikz-cd}
\usepackage[pdftex]{hyperref}
\usepackage{comment}
\newtheorem{lemma}{Lemma}[section]
\newtheorem{proposition}[lemma]{Proposition}
\newtheorem{corollary}[lemma]{Corollary}
\newtheorem{theorem}[lemma]{Theorem}
\newtheorem{non-theorem}{Non-Theorem}
\newtheorem{conjecture}{Conjecture}

\newtheorem{definition}[lemma]{Definition}
\newtheorem{remark}[lemma]{Remark}

\newtheorem{question}{Question}

\newcommand\abs[1]{|#1|}
\newcommand{\R}{\mathbb{R}}
\newcommand{\Z}{\mathbb{Z}}
\newcommand{\C}{\mathbb{C}}
\newcommand{\Vol}{\operatorname{Vol}}
\newcommand{\scr}{\mathcal}
\newcommand{\RP}{\mathbb{R}\mathbb{P}^2}
\newcommand{\inte}{\operatorname{int}}

\renewcommand{\Re}{\operatorname{Re}}
\renewcommand{\Im}{\operatorname{Im}}

\numberwithin{equation}{section}

\newcommand{\proofend}{\hspace*{\fill} $\Box$\\}

\def\1{\:\!}
\def\2{\;\!}

\def\reg{\operatorname{reg}}

\def\Vol{\operatorname {Vol}}

\def\aff{\operatorname{aff}}

\def\Sympc0{\operatorname{Symp^c_0}}

\def\Int{\operatorname{Int}}

\def\area{\operatorname{area}}

\def\C{\operatorname{C}}

\def\cd{{\mathcal D}}

\def\cs{{\mathcal S}}

\def\cx{{\mathcal X}}

\def\C{\mathbb{C}}

\def\N{\mathbb{N}}
\def\Q{\mathbb{Q}}
\def\R{\mathbb{R}}
\def\Z{\mathbb{Z}}

\def\RP{\operatorname{\mathbb{RP}}}
\def\CP{\operatorname{\mathbb{CP}}}

\def\pp{\partial}

\def\ni{\noindent}

\def\proof{\noindent {\it Proof. \;}}

\makeatletter
\let\@wraptoccontribs\wraptoccontribs
\makeatother

\begin{document}

\title[Quantifications of Gromov's non-squeezing theorem]{On certain quantifications of Gromov's non-squeezing theorem}

\author{Kevin Sackel}
\address{Department of Mathematics and Statistics, University of Massachusetts, Amherst,	MA 01003, USA}
\email{ksackel@umass.edu}

\author{Antoine Song}
\address{California Institute of Technology\\ 177 Linde Hall, \#1200 E. California Blvd., Pasadena, CA 91125}
\email{aysong@caltech.edu}

\author{Umut Varolgunes}
\address{Department of Mathematics, Bo\u{g}azi\c{c}i University}
\email{varolgunesu@gmail.com}

\author{Jonathan J. Zhu}
\address{Department of Mathematics, University of Washington, Seattle, WA, USA}
\email{jonozhu@uw.edu}

\address{Jo\'{e} Brendel,
	Institut de Math\'{e}matiques,
	Universit\'{e} de Neuch\^{a}tel}
\email{joe.brendel@unine.ch}

\contrib[with an appendix by]{Joe Brendel}

\begin{abstract}
Let $R>1$ and let $B$ be the Euclidean $4$-ball of radius $R$ with a closed subset ${E}$ removed. Suppose that $B$ embeds symplectically into the unit cylinder $\mathbb{D}^2 \times \mathbb{R}^2$. By Gromov's non-squeezing theorem, ${E}$ must be non-empty. We prove that the Minkowski dimension of ${E}$ is at least $2$, and we exhibit an explicit example showing that this result is optimal at least for $R \leq \sqrt{2}$. In an appendix by Jo\'{e} Brendel, it is shown that the lower bound is optimal for $R< \sqrt{3}$. We also discuss the minimum volume of ${E}$ in the case that the symplectic embedding extends, with bounded Lipschitz constant, to the entire ball. 
\end{abstract}
\maketitle

\section{Introduction}
%\subsection{Context}

Consider $\R^{2n}$ with its standard symplectic structure $\omega = \sum dx_i \wedge dy_i$. A prototypical question in symplectic geometry is to ask whether one domain of $\R^{2n}$ symplectically embeds into another (i.e. via an embedding $\Phi$ with $\Phi^*\omega = \omega$). At the very least, a symplectic embedding preserves the standard volume form $\frac{1}{n!}\omega^n$. However, there is more rigidity in symplectic geometry than just volume. We recall the most famous result certifying this bold claim:

Let $B^{2n}(\pi R^2)\subset \R^{2n}$ be the open ball of radius $R$, and let $Z^{2n}(\pi r^2) = B^2(\pi r^2) \times \R^{2n-2}\subset \R^{2n}$ be the open cylinder of radius $r$. 

\begin{theorem}[Gromov's Nonsqueezing Theorem \cite{Gro85}] There exists a symplectic embedding of $B^{2n}(\pi R^2)$ into $Z^{2n}(\pi r^2)$ if and only if $R \leq r$. 
\end{theorem}

Our goal in this paper is to try to quantify the failure of $B^{2n}(\pi R^2)$ to symplectically embed into $Z^{2n}(\pi r^2)$, when $R>r$, via the following motivating question:

\begin{center}
\fbox{\parbox{0.8\textwidth}{\begin{center}\textbf{Motivating Question:}\\
How much do we need to remove from $B^{2n}(\pi R^2)$ so that it embeds symplectically into $Z^{2n}(\pi r^2)$?
\end{center}}}
\end{center}

As a first attempt, one may try to use volume to answer this question. Over all possible symplectic embeddings $B^{2n}(\pi R^2) \hookrightarrow \R^{2n}$, what is the minimal possible volume excluded from $Z^{2n}(\pi r^2)$? This question has a straightforward answer, as the following result of Katok reveals: 
%, implies that there is a sequence of symplectic embeddings $B^{2n}(\pi R^2) \hookrightarrow \R^{2n}$ such that this volume is arbitrarily small.
\begin{theorem}[Katok \cite{Kat73}] \label{thm:Katok} 
Given a compact set $X$ in $(\R^{2n},\omega_{\mathrm{std}})$, for every $\epsilon > 0$, there exists a Hamiltonian diffeomorphism $\phi \colon (\R^{2n},\omega_{\mathrm{std}}) \rightarrow (\R^{2n},\omega_{\mathrm{std}})$ such that $$\mathrm{Vol}( \phi(X) \setminus Z(\pi) ) \leq \epsilon.$$
\end{theorem}

%%COME BACK

%\begin{theorem}[Katok \cite{Kat73}] \label{thm:Katok} Suppose we are given two compact subsets $X,Y$ of a symplectic manifold $(M,\omega)$, such that $\mathrm{vol}(X) \leq \mathrm{vol}(Y)$. Then for every $\epsilon > 0$, there exists a Hamiltonian diffeomorphism $\phi \colon (M,\omega) \rightarrow (M,\omega)$ such that $$\mathrm{vol}(Y \setminus \phi(X)) \leq \epsilon.$$
%\end{theorem}
%Taking $(M,\omega) = (\R^{2n},\omega_{\mathrm{std}})$, $X = B^{2n}(\pi R^2)$, and $Y = Z^{2n}(\pi r^2)$ yields the result relevant to this paper.

%This is in fact a particular case of Katok's theorem. The idea of proof is a symplectic piece-by-piece transportation procedure. 
%With $M$ as our ambient space, we want to symplectically transport as much of $X$ as we can into $Y$. 
%We cover almost all of $X$ with tiny disjoint boxes, with the property that these boxes can be independently Hamiltonian isotoped one by one so that they all fit disjointly inside of $Z(\pi)$. 
%The non-trivial part is determining how to apply this procedure in such a way that the boxes never collide with each other, so that the Hamiltonian isotopy is well-defined on the boxes at all times, and from which one can extend the Hamiltonian isotopy to all of $\R^{2n}$. 
%(In the particular instance we are interested in, where $M = \R^{2n}$ and $X$ is compact, this technicality is easily handled.)

We therefore modify our attempt to produce a more meaningful version of our question. We discuss two specific instances, though we focus predominantly on the first:
\begin{enumerate}
	\item What is the smallest Minkowski dimension of a subset ${E}\subset B^{2n}(\pi R^2)$ with the property that there is a symplectic embedding $B^{2n}(\pi R^2)\setminus {E}\to Z^{2n}(\pi r^2)$?
	\item Over all possible symplectic embeddings $B^{2n}(\pi R^2) \hookrightarrow \R^{2n}$ with Lipschitz constant at most $L>0$, what is the minimal possible volume excluded from $Z^{2n}(\pi r^2)$?
\end{enumerate}
Here Minkowski dimension stands for the lower Minkowski dimension, which is defined for any subset of $B^{2n}(\pi R^2)$. 
%The definition is provided in Section \ref{sectionprelim} below. 
Heuristically, ${E}$ having Minkowski dimension $d\in\mathbb{R}$ means that as $\epsilon\to 0$, the volume of the $\epsilon$-neighborhood of ${E}$ behaves as $c\epsilon^{2n-d}$, for some constant $c>0$.

For each of these two questions, there is some quantity we are trying to minimize: either the Minkowski dimension of ${E}$, or the volume excluded from the cylinder under an $L$-Lipschitz symplectic embedding. In either case, there are two key aspects to discuss.
\begin{itemize}
\item \textbf{Constructive}: We find an explicit symplectic embedding which provides an upper bound on the quantity that we are trying to minimize.
\item \textbf{Obstructive}: For a purported symplectic embedding, we find a lower bound on the quantity that we are trying to minimize.
\end{itemize}
A full answer to these two questions would require that the obstructive and constructive bounds match.
Although these questions are interesting in general dimensions, in this paper we will restrict our attention only to the case of dimension $2n=4$, save for a few open questions posed in Section \ref{sec:questions}.
%, mainly for the purpose of clarifying the main technical novelty of this paper. 
There are also a plethora of further questions that will be posed in the final section of the paper.

We now discuss our results for each of these two questions.

\subsection{The Minkowski dimension problem}

Recall that here we are asking for the smallest lower Minkowski dimension of a subset ${E}\subset B(\pi R^2)$ with the property that there is a symplectic embedding $B(\pi R^2)\setminus {E}\to Z(\pi r^2), $ assuming $R>r$. (Here and henceforth we drop the superscripts from $B^4(.)$ and $Z^4(.)$ since the dimension is always $4$.) By a scaling argument, it suffices to consider the case $r=1$.

Let us start by discussing the constructive side. Observe that if we remove a union of codimension one affine hyperplanes along a sufficiently fine grid, we end up with many connected components, each of which embeds into $ Z(\pi)$ by translations. Interestingly, at least in a certain range of $R$, one can do better and find a $2$-dimensional submanifold ${E}$, whose complement embeds into $Z(\pi)$. 

To explain this result we need to introduce some notation. Let us consider the Lagrangian disk 
$$L:=B(2\pi)\cap \{y_1=y_2=0\}.$$
%\begin{remark}Note that $U(n)=O(2n)\cap Sp(2n)$ acts transitively on linear Lagrangian planes in $\mathbb{R}^4$. \textcolor{red}{Do we really need this remark? - Kevin} \end{remark}
Let us also define $\mathcal{E}(\pi,4\pi)\subset \mathbb{R}^4$ to be the open ellipsoid $$\{(x_1,y_1,x_2,y_2)\mid x_1^2+y_1^2+\frac{x_2^2+y_2^2}{4}<1\},$$ and let $$\mathcal{C}:=\mathcal{E}(\pi,4\pi)\cap\{x_2=y_2=0\}.$$

\begin{theorem}\label{thmsymplectomorphism}
$B(2\pi)\setminus L$ is symplectomorphic to $\mathcal{E}(\pi,4\pi)\setminus \mathcal{C}$. 

Consequently, $B(2\pi)\setminus L$ symplectically embeds into $Z(\pi)$.
\end{theorem}

\begin{comment}
{\color{red} Add an explanation of what is done in the Appendix. Explain how conjecture below is modified accordingly. Say something like our Theorem 1.2 was also inspired by the toric degeneration picture of the Appendix but after (wrongly) convincing ourselves that the other degenerations would give worse embeddings by a mistake in basic arithmetic, we quickly set down to prove the result by different methods. This leads to a more explicit embedding and unlike Appendix we are able to determine the defect region precisely. This is meaningful from the viewpoint of the Minkowski content question but not so much for the dimension problem. In higher dimensions because vanishing cycles and symplectic divisors will have different dimensions, the situation is a bit more tricky}
\end{comment}

In particular, removing a Lagrangian plane from $B(2\pi)$ halves its Gromov capacity. Our proof of Theorem \ref{thmsymplectomorphism} has a great deal in common with the Section 3 of Oakley-Usher's beautiful paper \cite{OU16} where the same geometries are used for a different purpose. In fact, we show in Section \ref{sec:construction} how the projective space $\mathbb{CP}^2$ is symplectomorphic to the boundary reduction of the unit cotangent bundle $D^*\mathbb{RP}^2$ by using the explicit map of \cite[Lemma 3.1]{OU16}. Theorem \ref{thmsymplectomorphism} can also be derived from the proof of Biran's general decomposition theorem \cite[Theorem 1.A; Example 3.1.2]{biran}, and Opshtein \cite[Lemma 3.1]{opshtein}. We use the latter in our argument as well.

%We could not find Theorem \ref{thmsymplectomorphism} in the literature, but nothing that we use in the proof is original.

%\textcolor{red}{The following three remarks should go elsewhere! - Kevin}
%\begin{remark}
%In fact, it would be enough for our purposes if $L$ in the statement was simply a $2$ dimensional submanifold of $B(2\pi)$. We mention some relevant future work in Section *** where $L$ being a Lagrangian plane might play a role.
%\end{remark}

%\begin{remark}
%It is well-known that the $2$-width of $B(2\pi)$ is $2\pi$. The standard Gromov non-squeezing argument shows that as a result of this corollary $B(2\pi)-L$ admits a rank $2$ foliation all of whose leaves have metric area at most $\pi$. It is not hard to construct such a foliation directly. Namely, we take all the complex planes that intersect $L$ non-transversely...... explain in more detail.
%\end{remark}

%\begin{remark}As another corollary we deduce that removing $L$ from $B(2\pi)$ halves its Gromov capacity. On the other hand, removing a complex plane from $B(2\pi)$ does not change the Gromov capacity by the Traynor trick. See section ** for further discussion. \end{remark}

On the obstructive side, we show that removing a $2$-dimensional subset as in Theorem \ref{thmsymplectomorphism} is the best one can do in general:
\begin{theorem}\label{thm:minkowski_obstruct}
Let ${E}$ be a closed subset of $\mathbb{R}^4$ and let $R>1$. Suppose that $B(\pi R^2) \setminus {E}$ symplectically embeds into the cylinder $Z(\pi)$. Then the lower Minkowski dimension of ${E}$ is at least $2$.
\end{theorem}
For the proof of Theorem \ref{thm:minkowski_obstruct}, we build on Gromov's original non-squeezing argument by adding a key new ingredient: the waist inequality, which was also introduced by Gromov \cite{Gro03} (see also Memarian \cite{Mem11}). Crucially, we require the sharp version due to Akopyan-Karasev \cite{AK17}, as well as the Heintze-Karcher \cite{HK78} bound on the volumes of tubes around minimal surfaces, in place of the monotonicity inequality for minimal surfaces. 
%We prove a bound that might be optimal for all $R>1$ but for which we do not have a sharp construction for $R^2 > 2.$

%Combining Theorem \ref{thmsymplectomorphism},  Theorem \ref{thm:minkowski_obstruct}, we are thus able to fully answer our first question when the radius $R$ is between $1$ and $\sqrt{2}$. On the other hand, we do not know if the lower bound is optimal for  $R>\sqrt{2}$, see Subsection \ref{subsec:largeR}.

%\textcolor{red}{I feel like we should highlight the tools that we use here, and that this is in many ways the crux of the paper.} Hence for $R^2 \leq 2$, we have a full answer to our desired question. We will come back to the many follow-up questions in Section ***.
\begin{remark}\label{rem-joe} Let $R_{\mathrm{sup}}\in (1,\infty]$ be the supremum of the radii $R$ such that there is codimension $2$ subset of $B(\pi R^2)$ whose complement can symplectically embed into $Z(\pi)$.
In the first version of this article we had conjectured that $R_{\mathrm{sup}}$ should be equal to $\sqrt{2}$. However shortly after its appearance, Jo\'{e} Brendel informed us that using a construction inspired by \cite{HacPro10} he can prove $R_{\mathrm{sup}} \geq \sqrt{3}$. As a consequence, we changed our conjecture to a question, see Section \ref{subsec:largeR}. His construction appears in Appendix \ref{appx} below.

We further remark that in Theorem \ref{thmsymplectomorphism}, we remove a Lagrangian plane. In the construction of Jo\'e Brendel in Appendix \ref{appx} realizing $R_{\mathrm{sup}} \geq \sqrt{3}$, he removes a union of Lagrangians together with a symplectic divisor. In higher dimensions (see e.g. Section \ref{ssec:higher_dims}), this distinction could be interesting. \end{remark}

\subsection{The Lipschitz problem}

Recall that for fixed $L > 1$, we are asking for the smallest volume of the region
$${E}(\Phi) := B(\pi R^2) \setminus \Phi^{-1}(Z(\pi)) = \Phi^{-1}(\R^4\setminus Z(\pi))$$
over all symplectic embeddings $\Phi \colon B(\pi R^2) \hookrightarrow \R^4$ of Lipschitz constant bounded above by $L$. (We note that although we use the letter $L$ for both the Lipschitz constant as well as for the Lagrangian disk of Theorem \ref{thmsymplectomorphism}, there will be no confusion given the context.)

On the obstructive side, we obtain the following as a corollary of the proof of the obstructive bound for the Minkowski question (Theorem \ref{thm:minkowski_obstruct}):

\begin{theorem}
%[Volume defect for Lipschitz embeddings - obstructive bound] 
\label{thm:Lip_ob}
Let $R > 1$. Then there exists a constant $c = c(R) > 0$ such that for all constants $L$ and all symplectic embeddings $\Phi: B(\pi R^2)\hookrightarrow \mathbb{R}^4$ with Lipschitz constant at most $L$, we have
$$\Vol_4\big( {E}(\Phi)\big) \geq \frac{c}{L^2}.$$
\end{theorem}

\
It is worth noting that one may use the standard non-squeezing theorem alone to find a weaker quantitative obstructive bound of $c/L^3$ as follows. Suppose we had an $L$-Lipschitz symplectic embedding $\phi \colon B(\pi R^2) \hookrightarrow \R^{4}$ for $R > 2$. 
%Then by Gromov's non-squeezing theorem, there is at least on point $q = \phi(p)$ in the image landing outside of $Z^4(4\pi)$, and hence the ball $B'$ of radius $1$ centered at $q$ is completely disjoint from $Z^4(\pi)$. This means $\phi^{-1}(B') \subset D(\phi)$, and by the Lipschitz condition, $\phi^{-1}(B')$ contains the ball of radius $1/L$ around $p$ intersected with $B(\pi R^2)$. Hence
%
%$$\mathrm{Vol}(D(\phi)) \geq \mathrm{Vol}(\phi^{-1}(B')) \geq O(1/L^4).$$
%
%With a little more effort, one may find $O(L)$ such points with disjoint $O(1/L)$-neighborhoods which land outside of $Z^4(\pi)$, yielding the obstructive bound $O(1/L^3)$.
Then by Gromov's non-squeezing theorem and the Lipschitz condition, one can check that there is a ball of radius of order $1/L$ embedded inside ${E}(\phi)$. Hence,
$$\mathrm{Vol}(E(\phi)) \gtrsim 1/L^4.$$
With a little more effort, one may find order $L$ many disjoint such balls inside ${E}(\phi)$, yielding the obstructive bound $c/L^3$.
However, jumping from $c/L^3$ to our obstructive bound of $c/L^2$ appears to require a new tool, which in our case is Gromov's waist inequality.

On the constructive side, we adapt Katok's ideas in \cite{Kat73} to prove the following:
%The fact that there is any obstructive bound at all implies that Katok's principle cannot be applied while simultaneously keeping the Lipschitz constant bounded. So what can we obtain on the constructive side? We perform a simple modification of Katok's idea of piece-by-piece teleportation in a manner which keeps track of the Lipschitz constant. We cook up an explicit map where the region of $B(\pi R^2)$ which lands outside of $Z^4(\pi)$ is given by a $1/L$-thick neighborhood of a collection of affine hyperplanes. We can then slide neighboring regions past each other, so that all of the Lipschitz constant is encoded in what happens along the affine hyperplanes. By making these regions $1/L$-thick, we ensure the construction can be performed with Lipschitz constant $O(L)$:
%In particular, we find the following: 
%for this construction that $\mathrm{Vol}(D) = O(1/L)$.

\begin{theorem} \label{thm:Lip_con} Let $R > 1$. Then there exists a constant $C = C(R) > 0$ such that for all constants $L$, there exists a symplectic embedding $\Phi \colon B^4(R) \hookrightarrow \R^4$ with Lipschitz constant at most $L$ such that
$$\Vol_4\big({E}(\Phi)\big) \leq  \frac{C}{L}.$$
\end{theorem}

\begin{remark} \label{remark:schlenk}
	As was pointed out to us by Felix Schlenk, our construction is a simplified version of multiple symplectic folding \cite[Sections 3 and 4]{Schlenk}.
\end{remark}

One would obviously like to push the obstructive and constructive bounds together. 

\subsection*{Organization of the paper}
We start by recalling some definitions and known theorems in Section \ref{sectionprelim}, which are then applied in Section \ref{sec:obstruction} to prove an obstructive bound implying Theorem \ref{thm:minkowski_obstruct}. In Section \ref{sec:construction} we construct the symplectomorphism of Theorem \ref{thmsymplectomorphism}. The Lipschitz problem, including Theorems \ref{thm:Lip_ob} and \ref{thm:Lip_con}, are discussed in Section \ref{sec:Lipschitz}.  We also list several related questions in the final Section \ref{sec:questions}. In Appendix \ref{appx}, written by Jo\'e Brendel, the construction mentioned in Remark \ref{rem-joe} appears.

\subsection*{Acknowledgments}
We thank Michael Usher for a useful e-mail correspondence. We also thank Felix Schlenk and Jo\'{e} Brendel for their interest and helpful comments our paper. K.S. thanks Larry Guth for originally suggesting the version of this problem involving Lipschitz constants which provided the initial impetus for this project. U.V. thanks Grigory Mikhalkin for very useful discussions on Theorem \ref{thmsymplectomorphism} and also sketching a proof of Corollary \ref{cornonextend} that we ended up not using; and Kyler Siegel for a discussion regarding Section \ref{sscapquestion}.

K.S. was partially supported by the National Science Foundation under grant DMS-1547145. This research was conducted during the period A.S. served as a Clay Research Fellow. J.Z. was supported in part by the National Science Foundation under grant DMS-1802984 and the Australian Research Council under grant FL150100126.

\section{Preliminaries} \label{sectionprelim}

We use the usual asymptotic notation $f\in O(g)$ to mean $|f| \leq Cg$ for some constant $C$, and $f\in o(g)$ to mean that $\lim_{t\to 0} \frac{f(t)}{g(t)} =0$. We write $f\in \Theta(g)$ if $f\in O(g)$ and $g\in O(f)$.

Let $S\subset \mathbb{R}^n$ be any bounded subset. Let $N_t(S)$ denote the open $t$-neighbourhood of $S$ with respect to the standard metric. If $\Sigma$ is a compact submanifold (possibly with boundary), let $V_t(\Sigma)$ be the exponential $t$-tube of $\Sigma$, i.e. the image by the normal exponential map of the open $t$-neighbourhood of the zero section in the normal bundle of $\Sigma$ (which is endowed with the natural metric). 

We denote by $\Vol_n$ the Euclidean $n$-volume of a set and set $\alpha_l = \frac{\pi^{l/2}}{\Gamma(\frac{l}{2}+1)}$. Note that when $n$ is a natural number, $\alpha_n = \Vol_n(B^n)$ is precisely the Euclidean volume of the unit $n$-dimensional ball $B^n \subset\mathbb{R}^n$. For $s\geq 0$, the $s$-dimensional lower Minkowski content of $S$ is defined as 
$${\underline{\mathcal{M}}_{s}(S)} := \liminf_{t\to 0^+} \frac{\Vol_n(N_t(S))}{\alpha_{n-s} t^{n-s}}.$$

Note that the normalization is chosen so that if $\Sigma^k \subset\mathbb{R}^n$ is a closed $k$-dimensional submanifold, then $\underline{\mathcal{M}}_k(\Sigma) = \Vol_k(\Sigma)$ coincides with the Euclidean $k$-volume of $\Sigma$.

The lower Minkowski dimension of $S$ is defined as 
\begin{align*}
\underline{\dim}_{\mathcal{M}}(S) & : = \inf\{s>0; \quad {\underline{\mathcal{M}}_{s}(S)} =0  \}\\ & = \sup \{s\geq 0 ; \quad {\underline{\mathcal{M}}_{s}(S)} >0\}.
\end{align*}
There are similar notions of upper Minkowski dimension and upper Minkowski content, which we will not need in this paper since a lower bound of the lower Minkowski content implies by definition the same lower bound for the upper Minkowski content. There are also equivalent definitions using ball packings. Replacing $S$ by its closure $\bar{S}$ does not change the Minkowski upper/lower dimensions.

\subsection{Waist inequalities}

The waist inequality for round spheres proved by Gromov \cite{Gro03} and with more details by Memarian \cite{Mem11} was extended to the case of maps from Euclidean balls by Akopyan-Karasev \cite[Theorem 1]{AK17}. The proof of the latter immediately implies the following:
\begin{theorem}[Waist inequality]
\label{waist} 
For any positive integers $n,k$, there exists a continuous function $h_{n,k}:(0,\infty) \to \mathbb{R}$, such that $h_{n,k} \in o(t^k)$ and the following holds: for any continuous map $f:B^n \to \mathbb{R}^k$, there exists $y\in \mathbb{R}^k$ such that 
\begin{equation}
\label{eq:waist}
\Vol_n\left(N_t(f^{-1}(y))
 { \cap B^n } \right) \geq \alpha_{n-k} \alpha_k t^k - h_{n,k}(t).
\end{equation}
\end{theorem}

%Note that \cite[Theorem 1]{AK17} is stated as the estimate $\underline{\mathcal{M}}(f^{-1}(y)) \geq \alpha_{n-k}$ on the Minkowski content of $f^{-1}(y)$, which implies a similar estimate to (\ref{eq:waist}) for this fibre. Their proof actually guarantees the above estimate, which is uniform in $f$ (and $y$), as there is always a fibre whose $t$-neighbourhood can be compared to the $t$-neighbourhood of an equatorial sphere $\mathbb{S}^{n+1-k}\subset \mathbb{S}^{n+1}$. 

It will be useful for our application that the above estimate is uniform in $f$. This uniform estimate is indeed implied by the proof of \cite[Theorem 1]{AK17} as they compare the $t$-neighbourhood of a fibre to the $t$-neighbourhood of an equatorial unit sphere $S^{n+1-k}\subset S^{n+1}$ (cf. the second last equation of their proof, with correct normalisation). The latter is independent of $f$, and by explicit calculation one may verify that actually $h_{n,k} \in O( t^{k+2})$. 

We remark that the waist inequality for spheres \cite{Gro03, Mem11} describes a stronger property than the above statement, since it gives optimal bounds on all (not just small) neighborhoods of the big fiber. 

\begin{comment}
{\color{red} \begin{remark}
(We remark for the non-expert reader that) the waist inequality for the round spheres is a much stronger statement which gives optimal bounds on all (not just small) neighborhoods of the big fiber. (The extension of the result to the round ball is based on the elementary observation that the Archimedes map that projects the unit sphere in $\mathbb{R}^{n+1}$ to a round $(n-1)$-ball by taking the last $n-1$ coordinates is measure preserving and contracting.) This leads to a result that is enough for our purposes but in itself is non-optimal if we look beyond the $(n-k)$-Minkowski content of the big fiber.
\end{remark}}
{\color{blue} \begin{remark}
	The usual waist inequality for round spheres is much stronger than the waist inequality for round balls, in that it gives precise optimal bounds for arbitrary $t$-neighborhoods of one of the fibers. In the theorem above, on the other hand, we only know that $h_{n,k} \in o(t^k)$, but we do not know the optimal such function $h_{n,k}$ making the inequality hold for all $f$. Nevertheless, this is enough for our purposes, since we are only interested in finding a fiber with large $(n-k)$-Minkowski content, which in turn only requires understanding the behavior of the volume of its $t$-neighborhood as $t \rightarrow 0$.
\end{remark}}
\end{comment}

\subsection{Tubes around minimal submanifolds}

The Heintze-Karcher inequality \cite{HK78} estimates the volume of tubes around compact submanifolds. We need the case of minimal submanifolds in Euclidean space (covered by \cite[Theorem 2.3]{HK78} with $\delta=0$ and Remark 2 on page 453 in \cite{HK78}), which may be stated as follows:
\begin{theorem}[Heintze-Karcher inequality] \label{hk} 
For any positive integers $n,k$ and any smooth compact $k$-dimensional minimal submanifold $\Sigma^k\subset \mathbb{R}^n$ with boundary, for $t>0$ we have
\begin{equation}
\label{eq:hk}
\Vol_n(V_t(\Sigma)) \leq  \Vol_k(\Sigma) \alpha_{n-k} t^{n-k}.
\end{equation}
\end{theorem}

Theorem \ref{hk} is again a uniform estimate on the exponential $t$-neighbourhood; one may compare it to the statement that the (upper) Minkowski $k$-content of a closed submanifold $\Sigma^k\subset \mathbb{R}^n$ is $\Vol_k(\Sigma)$. The point of Theorem \ref{hk} is that the constant in this estimate does not depend on the minimal $k$-submanifold $\Sigma$ or $t$. The main ingredient for its proof is an estimate for the Jacobian determinant of the normal exponential, which in Euclidean space is 1 to lowest order. The following term is controlled by the mean curvature $\mathbf{H}$, which one may expect from the interpretation of mean curvature as the first variation of ($k$-) area. Consequently, for a general compact submanifold $\Sigma$, (\ref{eq:hk}) will have an error term of order $c (\max_\Sigma |\mathbf{H}|) t^{n-k+1}$. 

\begin{comment}
{\color{red}\begin{remark} (Again for the non-expert reader let us point that) to obtain the proof of our Theorem \ref{hk}, we are specifically using \cite[Theorem 2.3]{HK78} with $\delta=0$ and Remark 2 on two pages before that in the same reference.\end{remark}
}
\end{comment}

%\textcolor{red}{Okay I might be confused, but do we actually need this? It should be true asymptotically for \textbf{any} regular submanifold (Minkowski, whatever content coincides). In fact if we go through Heintze-Karcher I think we get an error bounded by $c(\sup |H|) t^{n-k+1}$ ??}

\subsection{Gromov foliation and maps into the cylinder}

For $r<R$, recall that $B(\pi R^2)$ and $Z(\pi r^2)$ denote respectively the open ball and open cylinder of radius $r,R$ in $\mathbb{R}^4$.
In this subsection, we give a slight modification of the holomorphic foliation argument of Gromov \cite{Gro85} in dimension $n=4$:
%\begin{proposition} \label{foliation!}
%Let $R,r>0$, $A>\pi r^2$ and $\epsilon'>0$. Let $E$ be a compact subset of $\mathbb{R}^4$ and let $\phi:B(\pi R^2)\to \mathbb{R}^4$ be a continuous map such that $\phi|_{B(\pi R^2) \setminus E}$ is a smooth symplectic embedding of $B(\pi R^2)\setminus  E$ into the cylinder $Z(\pi r^2)$. 
%%Suppose $\phi|_{\partial B(\pi R^2)}$ is transverse to $\partial Z(\pi r^2)$. 
%Then there exists a continuous map $f: B(\pi R^2)\to \mathbb{R}^2$ smooth on $B(\pi R^2) \setminus E$ such that: 
%\begin{itemize}
%\item $f$ has no critical points on $B(\pi R^2) \setminus E$,
%\item for any $y$, $f^{-1}(y)\cap \big(B(\pi R^2) \setminus N_{\epsilon'}(\partial B(\pi R^2) \cup E)\big)$ is a complex submanifold of Euclidean area less than $A$. 
%\end{itemize}
%\end{proposition}
%
%\textcolor{red}{Check the above, find citation. see comments below. 
%Maybe it is better to state this for $\phi$ and $f$ with domain $B(\pi R^2)-E$? It seems this is all that we need for the old sketch. Then we can just take a continuous extension of $f$ below. The old statement also says the second property holds on the interior of $B(\pi R^2)-E$? }

\begin{proposition} \label{foliation!2}
Let $R,r>0$. Let $E$ be a compact subset of $\mathbb{R}^4$ and let $\phi:B(\pi R^2) \setminus E \to \mathbb{R}^4$ be a smooth symplectic embedding into the cylinder $Z(\pi r^2)$. Let $U$ be the closure of an open neighbourhood of $\partial B(\pi R^2) \cup E$ in $\mathbb{R}^4$. 
Then there exists a smooth map $f: B(\pi R^2)\setminus U\to \mathbb{R}^2$ such that: 
\begin{itemize}
\item $f$ has no critical points on $B(\pi R^2)\setminus U$;
\item for all $y\in\mathbb{R}^2$, if $f^{-1}(y)\cap B(\pi R^2)\setminus U$ is non-empty, then it is a two dimensional complex submanifold of Euclidean area less than $\pi r^2$. 
\end{itemize}
\end{proposition}

\begin{proof}
The following argument is standard in the symplectic community. As mentioned before the statement of the proposition, the ideas are due to Gromov \cite{Gro85}, though more thorough analytic details may be found elsewhere, e.g. \cite{MS12}.

We define $A:=\pi r^2$ for brevity. Since $B(\pi R^2) \setminus U$ has compact closure in $B(\pi R^2) \setminus E$, the image of $B(\pi R^2) \setminus U$ under $\phi$ lands in $B^2(\pi r_0^2) \times [-K,K]^2$ for some large constant $K$ and $0<r_0<r$ (possibly depending upon $U$). Let $S^2(A)$ denote the 2-sphere with standard symplectic form scaled to have total area $A$, and let $T^2_K = (\R/4K\Z)^2$ be the 2-torus with symplectic form induced by the standard form on $\R^2$. Then we have a symplectic embedding $B^2(\pi r_0^2) \times [-K,K]^2 \subset S^2(A) \times T^2_K$, and upon composing with $\phi$, we arrive at a symplectic embedding, also denoted $\phi$ (by abuse of notation),
$$\phi \colon B(\pi R^2) \setminus U \hookrightarrow S^2(A) \times T^2_K.$$

Let $J_0$ denote the standard complex structure on $B(\pi R^2)$, and let $J_1$ denote the standard (split) complex structure on $S^2(A) \times T_K^{2}$. We pick a special almost complex structure $J_{\phi}$ on $S^2(A) \times T_K^2$ which incorporates $\phi$ by requiring that it satisfies the following three properties:
\begin{itemize}
	\item On the image of $\phi$, $J_{\phi} = \phi_*(J_0)$
	\item $J_{\phi} = J_1$ in a neighborhood of $\{\infty\} \times T^2_K$
	\item Everywhere, $J_{\phi}$ is compatible with the symplectic form
\end{itemize}

By \cite[Proposition 9.4.4]{MS12}, which encompasses standard 4-dimensional techniques, the evaluation map
$$\mathrm{ev} \colon \scr{M}_{0,1}(\beta,J_{\phi}) \rightarrow S^2(A) \times T^2_K$$
is a diffeomorphism, where $\scr{M}_{0,1}(\beta,J_{\phi})$ is the moduli space  of $J_{\phi}$-holomorphic spheres with one marked point and in the class $\beta$. Meanwhile, the map which forgets the marked point $\tau \colon \scr{M}_{0,1}(\beta,J_\phi) \rightarrow \scr{M}_{0,0}(\beta,J_\phi)$ is a smooth fibration, with fibers diffeomorphic to $S^2$. By positivity of intersections (see \cite[Theorem 2.6.3]{MS12}), each fiber of $\tau$, which is a $J_{\phi}$-holomorphic sphere, intersects $\{\infty\} \times T^2_K$ once and transversely, so by the implicit function theorem, we have a canonical diffeomorphism $g$ of $T^2_K$ with $\scr{M}_{0,0}(\beta,J_\phi)$. We therefore obtain a map $h \colon S^2(A) \times T^2_K \rightarrow T^2_K$ by setting $h(p) = x$ when the unique sphere through $p$ passes through $(\infty,x)$. It is clear by construction that the diagram \begin{center}
\begin{tikzcd}
\scr{M}_{0,1}(\beta,J_{\phi})\arrow{d}{\mathrm{for}}\arrow{r}{\mathrm{ev}}&S^2(A) \times T^2_K\arrow{d}{h}\\
\scr{M}_{0,0}(\beta,J_{\phi})\arrow{r}{g}& T^2_K
\end{tikzcd}
\end{center}commutes. In particular since the forgetful map is a smooth $S^2$ fiber bundle, $h$ is smooth and has no critical points.

Notice that since the image of $B(\pi R^2) \setminus U$ under $\phi$ is contained in a contractible subset of $S^2(A) \times T^2_K$, we have that the composition $h \circ \phi \colon B(\pi R^2) \setminus U \rightarrow T^2_K$ lifts to a map $f \colon B(\pi R^2) \setminus U \rightarrow \R^2$. This is the function $f$ we desired in the statement of the proposition, and we must now check it satisfies both of the desired properties.

The fact that $f^{-1}(y)$ is a complex submanifold is simply because it is by definition a subset of a $J_{\phi}$-holomorphic sphere, and $J_{\phi}$ is chosen to equal $\phi_*(J_0)$ on the image of $\phi$. Finally, the area bound comes from the fact that the area of $f^{-1}(y)$ is at most the symplectic area of the corresponding sphere (the one passing through $(\infty,[y])$), which is just $A$ since symplectic area is purely homological.
\end{proof}

\section{A quantitative obstruction to partial symplectic embeddings} \label{sec:obstruction}

%%FIGURE?

As usual, for $r<R$, $B(\pi R^2)$ and $Z(\pi r^2)$ refer to the open ball and open cylinder of radius $r,R$ respectively, in $\mathbb{R}^4$. %We use the usual notation $o(t^2)$ for a function that, after being divided by $t^2$, converges to zero as $t>0$ converges to $0$. We use subscripts, $o_{r,R}(t^2)$ to denote the dependence of this function on $r,R$.
The main estimate of this section is the following obstructive bound: 

\begin{theorem} \label{thm:obstruct}
	Let ${E}$ be a compact subset of $\mathbb{R}^4$ and suppose that $B(\pi R^2) \setminus {E}$ symplectically embeds into the cylinder $Z(\pi r^2) \subset \mathbb{R}^4$. Then there is a function $k_{R} \in o(t^2)$ such that for any $t>0$,
	$$\Vol_4(N_t({E})) \geq \pi^2 (R^2 - r^2)t^2 - k_{R}(t).$$
\end{theorem}
\begin{proof}
	Let $\Phi :B(\pi R^2) \setminus {E} \to Z(\pi r^2)$ be the symplectic embedding of the statement.
	Consider $0<t<\frac{R-r}{2}$ and take $0<\delta < t$. Later, we will send $\delta \to 0$.
	
	Let $U_\delta$ be the closure of  $N_\delta(\partial B(\pi R^2)\cup  {E})$ in $\mathbb{R}^4$ and let $\tilde{f}_\delta: B(\pi R^2)\setminus U_\delta \to \mathbb{R}^2$ be the map given by Proposition \ref{foliation!2}. 
	Note $B(\pi R^2) \setminus U_\delta = B(\pi(R-\delta)^2) \setminus \overline{N_\delta(E)}$.
	We then take $f_{\delta} :B(\pi(R-\delta)^2)\to \mathbb{R}^2$ to be any continuous extension of $\tilde{f}_\delta$. Since $f_{\delta}$ agrees with $\tilde{f}_\delta$ on $B(\pi(R-\delta)^2)\setminus U_\delta$, by the conclusions of Proposition \ref{foliation!2} we have for any $y\in \mathbb{R}^2$ that
	$$\Sigma_\delta := f_{\delta}^{-1}(y) \cap \big(B(\pi(R- \delta)^2) \setminus U_{\delta}\big)$$ is a minimal submanifold with area less than $A:=\pi r^2$. 
	
	The main idea of our proof is that the waist inequality guarantees a fibre with large volume neighbourhoods, but by the area bound (and the structure of tubes) this can only happen if the fibre accumulates near the exceptional set. Accordingly, a key component is the following covering claim:
	
	\textbf{Claim}: Let $f_{\delta,t}$ be the restriction of $f_{\delta}$ to the ball $B(\pi(R-2t)^2)$. Then 
	\begin{equation}
		\label{eq:mink-cover}
		N_t\left(f_{\delta,t}^{-1}(y)\right) \cap B(\pi(R-2t)^2) \subset V_t(\Sigma_\delta) \cup N_{\delta+t}(E). 
	\end{equation}

	Indeed, by definition of $\Sigma_\delta$ and the supposition $\delta<t$ we have that \[f_{\delta,t}^{-1}(y) \setminus \overline{N_\delta(E)} = \Sigma_\delta \cap B(\pi(R-2t)^2) \subset \Sigma_\delta.\] Now given any submanifold $\Sigma$, its $t$-neighbourhood $N_t(\Sigma)$ is always {contained in the union of} the tube $V_t(\Sigma)$ and the $t$-neighbourhood $N_t(\partial \Sigma)$ of its boundary. So since $\partial \Sigma_\delta \subset \partial U_\delta$, we have that \[N_t\left(f_{\delta,t}^{-1}(y) \setminus \overline{N_{\delta}(E)}\right) \subset N_t(\Sigma_{\delta}) \subset V_t(\Sigma_\delta)\cup N_t(\partial U_\delta).\] By the triangle inequality it follows that 
	\begin{equation}
		\label{eq:mink-fur} N_t(f_{\delta,t}^{-1}(y)) \subset V_t(\Sigma_\delta) \cup N_t(\partial U_\delta) \cup N_{\delta+t}(E).
	\end{equation}
	But by definition of $U_\delta$, we have $N_t(\partial U_\delta) \subset N_t(\partial B(\pi(R-\delta)^2)) \cup N_{\delta+t}(E)$, and since $\delta<t$, we note that $N_t(\partial B(\pi(R-\delta)^2)) \cap B(\pi(R-2t)^2))=\emptyset$. Taking the intersection of (\ref{eq:mink-fur}) with $B(\pi(R-2t)^2)$ then yields the claim. 
	
	Having established the claim, we now estimate the volume of each set in (\ref{eq:mink-cover}). First, let $h_{4,2}\in o(t^2)$ be as in Theorem \ref{waist}. By rescaling to the ball $B(\pi(R-2t)^2)$, the waist inequality Theorem \ref{waist} applied to $f_{\delta,t} : B(\pi(R-2t)^2) \to \mathbb{R}^2$ gives that there is some $y\in \mathbb{R}^2$ for which
	\begin{equation}
		\label{eq:mink-1}
		\Vol_4\left(N_t(f_{\delta,t}^{-1}(y)) \cap B(\pi(R-2t)^2)\right) \geq \pi^2 t^2 (R-2t)^2 - (R-2t)^4 h_{4,2}\left(\frac{t}{R-2t}\right).
	\end{equation}
	
	On the other hand, since $\Sigma_\delta$ is minimal with area at most $A$, the Heintze-Karcher inequality Theorem \ref{hk} yields 
	\begin{equation}
		\label{eq:mink-2}\Vol_4(V_t(\Sigma_\delta)) \leq   \Vol_2(\Sigma_\delta) \pi t^2 \leq A\pi t^2.
	\end{equation}

	Combining the covering (\ref{eq:mink-cover}) with the estimates (\ref{eq:mink-1}) and (\ref{eq:mink-2}) yields
	$$\Vol_4(N_{\delta +t}({E})) \geq (\pi^2 (R-2t)^2 - A \pi )t^2 - (R-2t)^4 h_{4,2}\left(\frac{t}{R-2t}\right).$$
	The  volume of $N_t({E})$ is non-decreasing with respect to $t$, and is continuous almost everywhere. 
	Therefore, sending $\delta \to 0$ and recalling that $A=\pi r^2$ in the inequality above, we obtain for any $t>0$ that:
	$$\Vol_4(N_{t}({E})) \geq \pi^2 ((R-2t)^2 - r^2 )t^2 - (R-2t)^4 h_{4,2}\left(\frac{t}{R-2t}\right).$$
	Thus taking \[k_R(t) = 4R\pi^2t^3  + R^4 h_{4,2}\left(\frac{t}{R-2t}\right) ,\] for instance, concludes the proof. 
\end{proof}

An immediate corollary is a lower bound for the lower Minkowski dimension. We will see that this bound is sharp in the next section, at least for radii $R$ which are not too large.
\begin{corollary}[Minkowski dimension] \label{cor:dimension}
	Suppose that $B(\pi R^2) \setminus {E}$ symplectically embeds into the cylinder $Z(\pi r^2) \subset \mathbb{R}^4$. 
	Then the $2$-dimensional lower Minkowski content of ${E}$ satisfies 
	$$\underline{\mathcal{M}}_{2} ({E}) \geq \pi(R^2-r^2).$$ In particular, the lower Minkowski dimension $\underline{\dim}_\mathcal{M} ({E})$ is at least $2$.
\end{corollary}

%Another corollary concerns the volume of the subset of $B(\pi R^2)$ sent outside of $Z(\pi r^2)$ by any Lipschitz symplectic embedding: 

%%MOVED THE COROLLARY TO LIPSCHITZ SECTION (and call it a theorem?)

\section{Squeezing the complement of a Lagrangian plane} \label{sec:construction}
%Proof of Theorem \ref{thmsymplectomorphism}

In this section it will be more convenient to use complex coordinates for the standard symplectic $\mathbb{R}^4$. Therefore we consider $\mathbb{C}^2$ with its standard K\"{a}hler structure, i.e. if $x$ and $y$ are the complex coordinates, then the symplectic form is $$\frac{i}{2}(dx \wedge d\bar{x}+dy\wedge d\bar{y}).$$

Let us recall the main objects in the statement of Theorem \ref{thmsymplectomorphism} in complex notation for convenience. Let $B(2\pi)\subset \mathbb{C}^2$ be the open ball of radius $\sqrt{2}$ centered at the origin. Let $\mathbb{R}^2\subset \mathbb{C}^2$ be the real part and define $$L:=B(2\pi)\cap \mathbb{R}^2.$$

We also define $\mathcal{E}(\pi,4\pi)\subset \mathbb{C}^2$ to be the open ellipsoid $$\{(x,y)\mid \abs{x}^2+\frac{\abs{y}^2}{4}<1\},$$ and let $$\mathcal{C}:=\mathcal{E}(\pi,4\pi)\cap\{y=0\}.$$

Let us introduce the main actors in the proof. Let $\mathbb{C}{P}^2(2\pi)$ be the symplectic manifold obtained from coisotropic reduction of the sphere $S^5$ of radius $\sqrt{2}$ in $\mathbb{C}^3$. Denoting the complex coordinates on $\mathbb{C}^3$ by $z_1,z_2,z_3$, we are using the symplectic structure $\frac{i}{2}(dz_1\wedge d\bar{z}_1+dz_2\wedge d\bar{z}_2+dz_3\wedge d\bar{z}_3)$ on $\mathbb{C}^3$. There is a canonical identification of $\mathbb{C}{P}^2(2\pi)$ with the complex manifold $\mathbb{C}{P}^2:=\mathrm{Gr}_{\mathbb{C}}(1,3)$, whose homogeneous coordinates we will denote by $[z_1:z_2:z_3]$. Of course, $\mathbb{C}{P}^2(2\pi)$ is nothing but $\mathbb{C}{P}^2$ equipped with the Fubini-Study symplectic form scaled so that a complex line (e.g. $\{z_1=0\}\subset \mathbb{C}{P}^2$) has area $2\pi$. 

Let us also specify some submanifolds of $\mathbb{C}{P}^2(2\pi)$ using its canonical identification with $\mathbb{C}{P}^2$. \begin{itemize}\item $L_{\RP}$ is the real part $$\{[z_1:z_2:z_3]\mid \Im(z_1)=\Im(z_2)=\Im(z_3)=0\}.$$ 
\item For  $t=[t_1:t_2:t_3]\in \RP:=Gr_{\mathbb{R}}(1,3)$, we define the complex lines $$S_t:=\{[z_1:z_2:z_3]\mid t_1z_1+t_2z_2+t_3z_3=0\}.$$
\item $FQ$ is the Fermat quadric $$\{[z_1:z_2:z_3]\mid z_1^2+z_2^2+z_3^2=0\}.$$
\end{itemize}
It is well-known that $\mathbb{C}{P}^2(2\pi)\setminus(S_{[0:0:1]}\cup L_{\RP})$ is symplectomorphic to $B(2\pi)\setminus L$ (see Exercise 9.4.11 in \cite{MS12}).

Consider $\RP$ as a smooth manifold in the standard way. Let $\lambda$ be the tautological one-form  on $T^*\RP$ and $V$ be the Liouville vector field, which is a vertical vector field equal to the Euler vector field in each fiber (which is defined on any vector space independently of a basis). We have $$\omega(V,\cdot)=\lambda,$$ where $\omega=d\lambda.$ We denote the zero section submanifold on $T^*\RP$ by $Z_{\RP}.$

The Riemannian metric on $S^2$ obtained from its embedding as the round sphere of radius $1$ in $\mathbb{R}^3$ induces a metric on its quotient by the antipodal map, which is canonically diffeomorphic to $\RP$. We call this the round metric on $\RP$ and denote it by $g_{\RP}$. We have the diffeomorphism $g_{\RP}^\sharp :T\RP\to T^*\RP,$ which is in particular linear on the fibers. We can transport the function $K:T\RP\to \mathbb{R},$ given by lengths of tangent vectors to 
\begin{equation}\label{ksharp}
K^\sharp: T^*\RP\to \mathbb{R}.
\end{equation}
On $T\RP$ we have the geodesic flow; under the identification by $g_{\RP}^\sharp$ this becomes the Hamiltonian flow of the function $(K^\sharp)^2/2$. The normalized geodesic flow on $T\RP \setminus \RP$ becomes the Hamiltonian flow of $K^\sharp$. Let us call these the geodesic flow and normalized geodesic flow on $T^*\RP$. Note that the normalized geodesic flow on $T^*\RP\setminus Z_{\RP}$ is a $\pi$-periodic action of $\mathbb{R}$.

Any unparametrized oriented geodesic circle $\gamma$ in $\RP$ with its round metric defines a symplectic submanifold $C_{\gamma}$ in $T^*\RP$ with boundary on $Z_{\RP}$ by taking points $(q,p)$ such that $q\in\gamma$ and $p=g_q(v,\cdot)$ where $v$ is non-negatively tangent to $\gamma$. Let us denote by $-\gamma$ the geodesic circle with opposite orientation. Clearly, $C_{\gamma}$ and $C_{-\gamma}$ intersect along $\gamma\subset Z_{\RP}$ and form a symplectic submanifold of $T^*\RP$, which is diffeomorphic to a cylinder. 

Let $D^*\RP\subset T^*\RP$ be the closed unit disk bundle, which is given by the subset $K^\sharp\leq 1$. Also let $U^*\RP = (K^\sharp)^{-1}(1)$ be the unit sphere bundle, i.e. the boundary of $D^*\RP,$ with its induced contact structure $\theta:=\iota^*\lambda$. 

The symplectic reduction $U^*\RP/S^1$ is a two-sphere equipped with a canonical symplectic form. Let us call this symplectic manifold $(Q,\omega_Q).$ The points of $Q$ are canonically identified with unparametrized oriented geodesic circles in round $\RP.$

Let us denote the boundary reduction symplectic manifold of $D^*\RP$ by $\overline{D^*\mathbb{RP}^2}$ (see Definition 3.9 of \cite{Sym03}; also note the interpretation as one half of a symplectic cut \cite{lerman}).

Note that $Q$ and $Z_{\RP}$ sit naturally inside $\overline{D^*\mathbb{RP}^2}$.  The Poincar\'{e} dual of the homology class of $Q$ is $\frac{1}{\pi}$ times the symplectic class. The cylinders $(C_{\gamma}\cup C_{-\gamma})\cap D^*\RP$ become symplectic 2-spheres in $\overline{D^*\mathbb{RP}^2}$. They intersect $Q$ positively in two points and $Z_{\RP}$ along the circle $\gamma$. Let us call these spheres $S_{\gamma}$, now indexed by unoriented unparametrized geodesic circles on $\RP$. Each $S_{\gamma}$ has self-intersection number $1$.

Let $\gamma_0$ be the oriented unparametrized geodesic on $\RP$ which corresponds to the quotient of the horizontal great circle in $S^2\subset \mathbb{R}^3$ oriented  as the boundary of the lower hemi-sphere. We define $S:=S_{\gamma_0}$. 

\begin{proposition} \label{extension}
There is a symplectomorphism $\overline{D^*\mathbb{RP}^2}\to \mathbb{C}\mathbb{P}^2(2\pi)$ with the following properties: \begin{itemize}
\item $Z_{\RP}$ is sent to $L_{\RP}.$
\item $S$ is sent to $S_{[0:0:1]}$.
\item $Q$ is sent to $FQ$.
\end{itemize}
\end{proposition}
The proof of this proposition is postponed to Subsection \ref{subsection: extension}.
% It will suffice to show that the symplectomorphism $$\inte(D^*\RP)\to \mathbb{C}\mathbb{P}^2(2\pi)-FQ$$ of \cite{OU16} Lemma 3.1 extends to a diffeomorphism $\overline{D^*\mathbb{RP}^2}\to \mathbb{CP}^2$. 
Let us continue with an immediate corollary.

\begin{corollary} \label{utile}
$\overline{D^*\mathbb{RP}^2}\setminus (S\cup Z_{\RP})$ is symplectomorphic to $\mathbb{C}{P}^2(2\pi)\setminus (S_{[0:0:1]}\cup L_{\RP})$, and in turn to $B(2\pi)\setminus L$.
\end{corollary}

Note that $V\cdot K^\sharp=K^\sharp$ on $T^*\RP\setminus Z_{\RP}$, which means that $K^\sharp$ is an exponentiated Liouville coordinate for $V$ on $T^*\RP\setminus Z_{\RP}$. Hence, we obtain a Liouville isomorphism $$T^*\RP\setminus Z_{\RP}\simeq (U^*\RP\times (0,\infty)_r, d(r\theta)),$$ where $ K^\sharp$ is matched with the function $r.$ The Hamiltonian vector field $X_r$ gives the $r$-translation invariant Reeb vector field (using $r=1$) on the contact levels. Finally, observe that $C_{\gamma}\cap(T^*\RP\setminus Z_{\RP})$'s are obtained as the traces of the Reeb orbits on $U^*\RP$ under the Liouville flow.

In particular, we have a foliation of $\overline{D^*\mathbb{RP}^2}\setminus Z_{\RP}$ by open disks which are the reductions of $C_{\gamma}\cap(D^*\RP\setminus Z_{\RP})$'s. Let us denote this $Q$-family of submanifolds by $\mathbb{D}_\gamma$, $\gamma\in Q$.

\begin{proposition}\label{propbrsd}
$\overline{D^*\mathbb{RP}^2}\setminus Z_{\RP}$ is symplectomorphic to an area $\pi$ standard symplectic disk bundle of $(Q,\omega_Q)$ in the sense of Biran (Section 2.1 \cite{biran}) in such a way that $\mathbb{D}_\gamma$ is sent to the fiber over $\gamma$ for every $\gamma\in Q$.
\end{proposition}

\begin{remark}
Note that the cohomology class of $\frac{\omega_Q}{\pi}$ is integral and it admits a unique lift to $H_2(Q,\mathbb{Z})$. What we mean by an area  $\pi$ standard symplectic disk bundle of $(Q,\omega_Q)$ is an area  $1$ standard symplectic disk bundle of $(Q,\frac{\omega_Q}{\pi})$ with its symplectic form multiplied by $\pi$.
\end{remark}

\begin{proof}
The symplectomorphism $D^*\RP\setminus Z_{\RP}\simeq U^*\RP\times (0,1]_r$ induces a symplectomorphism of the boundary reductions of both sides.  We note that
\begin{align}\label{eqsympbr}d(r\tilde{\theta})=\underline{\mathrm{pr}}^*\omega_Q-d((1-r)\tilde{\theta})\end{align} on $U^*\RP\times (0,1)$, where we define $\tilde{\theta}=\mathrm{pr}^*\theta$ for clarity.% and $\underline{\mathrm{pr}}$ is the composition of the symplectic reduction map $U^*\RP\to Q$ with $\mathrm{pr}: U^*\RP\times (0,1)\to U^*\RP$.

Here we use the maps in the following commutative diagram, where $\mathrm{Pr},\mathrm{pr}$ are the obvious projections, and $U^*\mathbb{RP}^2\to Q$ is the symplectic reduction map.
\begin{center}
\begin{tikzcd}
U^*\RP\times [0,\infty)_{\rho} \arrow{dr}{\mathrm{Pr}}  \arrow{ddr}[swap]{\underline{\mathrm{Pr}}} &&U^*\RP\times (0,1)_r \arrow{dl}[swap]{\mathrm{pr}} \arrow{ddl}{\underline{\mathrm{pr}}}\\
& U^*\mathbb{RP}^2 \arrow{d}{} &\\
&Q&
\end{tikzcd}
\end{center}

Notice that the map  $U^*\RP\to Q$ has the structure of a principal $U(1)=\mathbb{R}/\mathbb{Z}$ bundle structure using the Reeb flow of the contact form $\frac{\theta}{\pi}$, and the associated complex line bundle $\mathcal{L}$ is precisely the fiberwise blow-down of $U^*\RP\times [0,\infty)_{\rho}$ with respect to its canonical projection $\underline{\mathrm{Pr}}$ to $Q$. The integral Chern class of this complex line bundle is $\frac{\omega_Q}{\pi}$ (using that $H_*(Q,\mathbb{Z})$ has no torsion) and a transgression 1-form is given by the pull-back of $-\frac{\theta}{\pi}$ by $\mathrm{Pr}$. %:U^*\RP\times (0,\infty)_{\rho}\to U^*\RP.$ 
By definition, the open unit disk bundle $\rho<1$ inside $\mathcal{L}$, endowed with the symplectic form \begin{align}\label{eqsympsd}\underline{\mathrm{Pr}}^*\left(\frac{\omega_Q}{\pi}\right)+d\left(\rho^2\mathrm{Pr}^*\left(-\frac{\theta}{\pi}\right)\right),\end{align} is an area  $1$ standard symplectic disk bundle of $(Q,\frac{\omega_Q}{\pi})$. Therefore, if we multiply this form by $\pi$, we obtain an area  $\pi$ standard symplectic disk bundle of $(Q,\omega_Q)$.

Now consider the following commutative diagram:
\begin{center}
\begin{tikzcd}
U^*\RP\times [0,1)_{\rho}\arrow{rr}{r=1-\rho^2}\arrow{d}{}&&U^*\RP\times (0,1]_r \arrow{d}{}
\\ \{\rho<1\}\subset \mathcal{L}\arrow{rr}{F}&&\overline{U^*\RP\times (0,1]}
\end{tikzcd}
\end{center}

Here the left map is the restriction of the fiberwise blowdown map and the right map is the boundary reduction map. The upper map sends $(x,\rho) \in U^*\RP\times [0,1)_{\rho}$ to $(x,1-\rho^2)$, and it is a homeomorphism overall as well as a diffeomorphism of the interiors. By construction of symplectic boundary reduction we deduce that there is a canonical diffeomorphism $F$ making this diagram commutative. 

It automatically follows from comparing Equations \eqref{eqsympbr} and \eqref{eqsympsd} that $F$ is a symplectomorphism. Composing $F$ with the symplectomorphism from the very beginning of this proof yields the desired symplectomorphism.

\end{proof}

%\begin{corollary}
%The homology class of $Q$ is two times the homology class of $S_\gamma$. Therefore, $Q$ has self-intersection $4$, and symplectic area $4\pi$.
%\end{corollary}

Combined with Corollary \ref{utile}, the following finishes the proof of Theorem \ref{thmsymplectomorphism}.

\begin{proposition}
$\overline{D^*\mathbb{RP}^2}\setminus (S\cup Z_{\RP})$ is symplectomorphic to $\mathcal{E}(\pi,4\pi)\setminus \mathcal{C}$.
\end{proposition}

\begin{proof}We use our Proposition \ref{propbrsd} and Lemma 2.1 of \cite{opshtein} to find an explicit symplectomorphism from the complement of $\mathbb{D}_{\gamma_0}$ in $\overline{D^*\mathbb{RP}^2}\setminus Z_{\RP}$ to $\mathcal{E}(\pi,4\pi)$. Here we use a symplectomorphism between $Q\setminus \{\gamma_0\}$ and the two-dimensional open ellipsoid of area $4\pi$ which sends $-\gamma_0$ to the origin, so that $\mathbb{D}_{-\gamma_0}$ is sent to $\mathcal{C}$ by this symplectomorphism.
\end{proof}

\subsection{Proof of Proposition \ref{extension}} \label{subsection: extension}

We will freely use the canonical identification of $\mathbb{CP}^n(2\pi)$ with $(\mathbb{C}^{n+1}\setminus \{0\})/\mathbb{C}^*$. Note that the homogenous coordinates $[z_1:\cdots: z_{n+1}]$ denote the class $[z_1e_1+\cdots + z_{n+1} e_{n+1}]$, where $e_i$ is the standard basis of $\mathbb{R}^{n+1}$ and its complexification $\mathbb{C}^{n+1}$. We also realize $T^*\mathbb{RP}^n$ as $\{(q,p)\in S^n\times  \mathbb{R}^{n+1}, \langle q,p\rangle =0\} / \{\pm 1\}$, where $S^n\subset  \mathbb{R}^{n+1}$ is the unit sphere. It is a straightforward computation that the standard symplectic form on $T^*\mathbb{RP}^n$ descends from the restriction of $\sum_{i=1}^{n+1}dp_i \wedge dq_i$ on $\mathbb{R}^{2n+2}$ under this identification. Note also that $K^\sharp([q,p]) = |p|$ away from the zero section.

In \cite[Lemma 3.1]{OU16}, Oakley and Usher considered the map $$\Phi: D^*\mathbb{RP}^n \to \mathbb{CP}^n(2\pi)$$ defined by 
\begin{equation}
\label{eq:oakleyusher}
\Phi([q,p]) := [ \sqrt{f(|p|)} p + \frac{i}{\sqrt{f(|p|)}} q],
\end{equation}
where $f(x) = \frac{1-\sqrt{1-x^2}}{x^2}$ on $(0,1]$ and $f(0)=\frac{1}{2}$. They proved that $\Phi|_{\inte(D^*\mathbb{RP}^n)}$ is a symplectomorphism onto its image $\mathbb{CP}^n(2\pi)\setminus FQ_n$, where $$FQ_n= \{ [z_0:\cdots :z_n] \in \mathbb{CP}^n(2\pi) | \sum z_k^2=0\}$$ is the Fermat quadric. As before, we will denote by $\overline{D^*\mathbb{RP}^n}$ the boundary reduction of ${D^*\mathbb{RP}^n}$, and by $Z_{\mathbb{RP}^n}$ the zero section.
We have the following which implies Proposition \ref{extension}:

\begin{proposition}
The Oakley-Usher map $\Phi: D^*\mathbb{RP}^n \to \mathbb{CP}^n(2\pi)$ descends to a symplectomorphism
$$\overline{\Phi} : \overline{D^*\mathbb{RP}^n} \to \mathbb{CP}^n(2\pi).$$
\end{proposition}

\begin{proof}

Note that $\overline{D^*\mathbb{RP}^n}$ is canonically homeomorphic to the quotient $D^*\mathbb{RP}^n/\sim$, where $x\sim y$ if $x,y\in U^*\mathbb{RP}^n$ and they are in the same orbit of the geodesic flow. Therefore, it is easy to see from the computations of Oakley-Usher that $\Phi$ descends to a bijective continuous map $\overline{\Phi} : \overline{D^*\mathbb{RP}^n} \to \mathbb{CP}^n(2\pi).$ We will show that this map is a symplectomorphism. 

To do so, it suffices to show that $\overline{\Phi}$ is smooth. Indeed, if $\overline{\Phi}$ is smooth, then by continuity it follows that $\overline{\Phi}$ preserves the symplectic form. This in particular shows that $\overline{\Phi}$ is an immersion, and hence a diffeomorphism that preserves the symplectic forms.

The following point is crucial. The canonical (linear) action of the group $G=\mathrm{SO}(n+1)$ on $\mathbb{R}^{n+1}$ induces an action on $D^*\mathbb{RP}^n$ (and in turn on $\overline{D^*\mathbb{RP}^n}$) and $\mathbb{C}^{n+1}$ (and in turn on $\mathbb{CP}^n(2\pi)$). It is clear from the definition (\ref{eq:oakleyusher}) that $\overline{\Phi}$ is $G$-equivariant with respect to these actions. 

We first prove smoothness in the case $n=1$. Equip  $\mathbb{CP}^1(2\pi)$ with the induced Riemannian metric (the so called Fubini-Study metric), which makes it isometric to a round $S^2$. We will use the fact that the image of a linear Lagrangian subspace in $\mathbb{C}^2\setminus \{0\}$ (with standard K\"{a}hler structure) under the canonical projection to  $\mathbb{CP}^1(2\pi)$ is a geodesic circle. We will denote  $\mathbb{CP}^1(2\pi)$ by $\mathbb{CP}^1$ for brevity.

Note that $Z_{\mathbb{RP}^1}$ is sent to the real part $L_{\mathbb{RP}^1}\subset \mathbb{CP}^1$ under $\overline{\Phi}$. Moreover, $\overline{D^*\mathbb{RP}^1}\setminus  \inte(D^*\mathbb{RP}^1)$ consists of two points which map to $[1:\pm i]$. Finally, notice that the images of cotangent fibres (that is, line segments of constant $q$) are sent to geodesic segments connecting $[1: i]$ and $[1:- i]$. It is easy to see that these geodesics are orthogonal to the geodesic circle $L_{\mathbb{RP}^1}$. Also recall that $\overline{\Phi}$ is $\mathrm{SO}(2)$-equivariant as explained above.

Let $(D^*\mathbb{RP}^1)^+:=\{[q,p]\, |\, p_1q_2-p_2q_1\leq 0\}\subset D^*\mathbb{RP}^1.$ Note that $K^\sharp(q,p) = |p|$ is a smooth function when restricted to $(D^*\mathbb{RP}^1)^+$. On $(D^*\mathbb{RP}^1)^+$ the symplectic form is easily computed to be $dK^\sharp \wedge d\theta$, where $\theta: (D^*\mathbb{RP}^1)^+\to \mathbb{R}/\pi\mathbb{Z}$ is defined by the relation $[q,p]=[\cos(\theta),\sin(\theta),-\sin(\theta)|p|,\cos(\theta)|p|]$.

Let $\mathbb{D}\subset \mathbb{R}^2$ be the unit disk $|x|^2+|y|^2\leq 1$ with the symplectic form $dx\wedge dy$. It is well-known that there is a symplectic embedding  $\mathbb{D}\hookrightarrow \mathbb{CP}^1$ which sends:
\begin{itemize}
\item the origin to $[1:i]$;
\item the unit circle to $L_{\mathbb{RP}^1}$;
\item radial rays to the geodesic segments that connect  $L_{\mathbb{RP}^1}$  and $[1:i]$, in a way that preserves angles at the intersections;
\item finally, the disks centered at the origin to balls around $[1:i]$ (in the Fubini-Study metric). 
\end{itemize}

We now consider the induced continuous map $\Pi: (D^*\mathbb{RP}^1)^+\to \mathbb{D}$ defined by the following commutative diagram:

\begin{center}
\begin{tikzcd}
     (D^*\mathbb{RP}^1)^+ \arrow{r}{\Phi} \arrow{rd}[swap]{\Pi} & \mathbb{CP}^1 \\
    & \mathbb{D} \arrow[hookrightarrow]{u}{}
\end{tikzcd}
\end{center}

 Let $(\rho,\phi)$ denote polar coordinates on $\mathbb{D}\setminus \{0\}$. We can deduce from our discussion thus far the following facts:
\begin{itemize}
\item There exists a constant $c$ such that $\Pi^*\phi=-2\theta+c$;
\item There exists a function $h: [0,1)\to (0,1]$ such that $$ \Pi^*\rho=h\circ K^\sharp;$$
\item $\pi$ restricts to a symplectomorphism $(D^*\mathbb{RP}^1)^+\setminus U^*\mathbb{RP}^1 \to \mathbb{D}\setminus \{0\}$. 
\end{itemize}

These facts imply that $h$ satisfies the differential equation $h'(x)h(x)=-1/2$, which, with the initial condition $h(0)=1$, has the unique solution $h(x)=\sqrt{1-x}$.\footnote{This form of $h$ can also be recovered from the explicit forms of $\Phi$ and the embedding $\mathbb{D}\hookrightarrow \mathbb{CP}^1$. } We thus observe that the map $\Pi: (D^*\mathbb{RP}^1)^+\to \mathbb{D}$ is a model for the boundary reduction of $(D^*\mathbb{RP}^1)^+$ at $U^*\mathbb{RP}^1$ (see Equation (3.1) of \cite{Sym03}). This proves that the map $\overline{\Phi}$ is smooth at both points of $\overline{D^*\mathbb{RP}^1}\setminus  \inte(D^*\mathbb{RP}^1)$ as we can repeat the same argument on the other half of $\overline{D^*\mathbb{RP}^1}$.

We now move on to the case $n>1$. We start with a preliminary lemma.

\begin{lemma}
Let $G$ be a Lie group acting smoothly on $M$ and $N$. Suppose that $S$ is a smooth submanifold of $M$ and that the multiplication map $G\times S\to M$ is a surjective submersion. If $\phi: M\to N$ is a $G$-equivariant map and $\phi|_S$ is smooth, then $\phi$ is also smooth. 
\end{lemma}
\begin{proof}
We have the following commutative diagram in which each map is known to be smooth except the bottom map. 
\begin{center}
\begin{tikzcd}

     G\times S \arrow{r}{\mathrm{id}\times \phi} \arrow{d}{\mu_M} & G\times N \arrow{d}{\mu_N} \\
     M \arrow{r}{\phi} & N
\end{tikzcd}
\end{center}

Since the left map is smooth surjective submersion, it has local smooth sections $ M\supset U \to G\times S$. The commutativity then implies that $\phi$ is smooth. 
\end{proof}

Note that the orbits of $G=\mathrm{SO}(n+1)$ on $\overline{D^*\mathbb{RP}^n}$ are the submanifolds of constant $|p|$. Fix any unoriented geodesic circle $\gamma$ in $\mathbb{RP}^n$ with its round metric $g$. We obtain an embedding of $T^*\mathbb{RP}^1$ in $T^*\mathbb{RP}^n$ by taking points $(q,p)$ such that $q\in\gamma$ and $p=g_q(v,\cdot)$ where $v$ is tangent to $\gamma$. This restricts to a smooth embedding of $D^*\mathbb{RP}^1$ into $D^*\mathbb{RP}^n$, and of $\overline{D^*\mathbb{RP}^1}$ into $\overline{D^*\mathbb{RP}^n}$ (the last point is particularly clear in the description of the boundary reductions at hand as in the proof of Proposition \ref{propbrsd}, which we keep in mind for the next point as well.) 

It is easy to see that the multiplication map $G\times \overline{D^*\mathbb{RP}^1} \to \overline{D^*\mathbb{RP}^n}$ is indeed a surjective submersion using that $G\times T^*\mathbb{RP}^1 \to T^*\mathbb{RP}^n$ is one. Applying the lemma with $S= \overline{D^*\mathbb{RP}^1}$, smoothness of $\overline{\Phi}$ follows from the smoothness of the $n=1$ case $\overline{\Phi}_1: \overline{D^*\mathbb{RP}^1}\to \mathbb{CP}^1$.

\end{proof}

\begin{remark}
In \cite[Chapter V]{Sea06}, Seade gives a description of $\mathbb{CP}^n$ as a double mapping cylinder via the natural $\mathrm{SO}(n+1)$ action. One may follow this discussion to obtain the corresponding description of $\overline{D^*\mathbb{RP}^n}$, and that the map $\Phi$ factors as the normal exponential map of $\mathbb{RP}^n\hookrightarrow \mathbb{CP}^n$ (with respect to the Fubini-Study metric) composed with the map $D^*\mathbb{RP}^n \to T^*\mathbb{RP}^n \simeq N\mathbb{RP}^n$ induced by $|p| \mapsto \frac{1}{2}\sin^{-1}|p|$; note that $xf(x) = \tan(\frac{1}{2}\sin^{-1} x)$ and that the focal set of $\mathbb{RP}^n$ is precisely $FQ_n$. This yields an alternative construction of $\Phi$ as well as its extension $\overline{\Phi}$.
\end{remark}

\section{Lipschitz symplectic embeddings of balls} \label{sec:Lipschitz}

In this section, we aim to prove Theorems \ref{thm:Lip_ob} and \ref{thm:Lip_con} from the introduction. To begin we introduce some slightly more general notation, in which we also vary the radius of the cylinder. Suppose that $\Phi \colon B(\pi R^2) \rightarrow \R^4$ is a symplectic embedding with Lipschitz constant $L > 0$. Then we may set
$${E}(\Phi,r) := \Phi^{-1}(\R^4\setminus Z(\pi r^2)).$$
Recall now that Theorem \ref{thm:Lip_ob} is the statement that $\Vol_4\big({E}(\Phi,1)\big) \geq \frac{c}{L^2}$ for some constant $c = c(R) > 0$.

\begin{proof}[Proof of Theorem \ref{thm:Lip_ob}]
Let $\delta$ be any number strictly between $0$ and $R-1$. Observe that by the Lipschitz bound, we have
$$N_{\delta/L}({E}(\Phi,1+\delta)) \subset {E}(\Phi,1).$$

%%%HERE
Applying Theorem \ref{thm:obstruct} with 
$${E}= {E}(\Phi,1+\delta)$$
to the symplectic embedding
$$\Phi|_{B(\pi R^2) \setminus {E}(\Phi,1+\delta)} : B(\pi R^2) \setminus {E}(\Phi,1+\delta) \to Z(1+\delta).$$ 
We obtain
$$\Vol_4({E}(\Phi,1)) \geq \pi^2 (R^2 - (1+\delta)^2)\frac{\delta^2}{L^2} - o\left(\frac{\delta^2}{L^2}\right)$$
which implies the desired bound after fixing some value for $\delta$, for instance $\delta = \frac{R-1}{2}$.
\end{proof}

\begin{remark} \label{rmk:scaling}
We may more generally ask about the volume of the region ${E}(\Phi,r)$ for $R > r$. By a scaling argument, we find that
$$\mathrm{Vol}_4({E}(\Phi,r)) \geq \frac{r^4c(R/r)}{L^2}.$$
Hence, we lose no information by restricting to the case $r=1$.
\end{remark}

\begin{remark} \label{rmk:constant}
The proof demonstrates that we may take $c(R) \sim R^4$ as $R$ grows large.
\end{remark}

On the other hand, we wish to prove the constructive bound, in which we must find an embedding $\Phi \colon B^4(R) \rightarrow \R^4$ of Lipschitz constant bounded by $L$ such that
$$\Vol_4\big({E}(\Phi)\big) \leq \frac{C}{L}$$
for some $C = C(R) > 0$. (Recall here ${E}(\Phi) = {E}(\Phi,1)$.)

\begin{proof}[Proof of Theorem \ref{thm:Lip_con}]
The fact that $\Vol_4\big({E}(\Phi)\big) $ can be made arbitrarily small if there is no restriction on the Lipschitz constant is exemplified by ideas of Katok \cite{Kat73}. Our proof is a quantitative refinement obtained using constructions which appear in symplectic folding. The basic idea is to break the ball $B(\pi R^2)$ into a number of cubes and Hamiltonian isotope each of these cubes into $Z(\pi)$, where the cubes are separated by walls of width $1/L$. To begin, we make three simplifications. First, we replace the domain $B(\pi R^2)$ with the cube $K(R) = [-R,R]^4$, as the volume defect can only increase in size. Second, we replace $K(R)$ with the rectangular prism
$$K'(R) = ([0,4R^2] \times [0,1]) \times ([0,4R^2] \times [0,1]),$$
where the parentheses indicate a symplectic splitting. Explicitly, the factors refer to the coordinates $x_1$, $y_1$, $x_2$, $y_2$, with symplectic form $\omega = dx_1 \wedge dy_1 + dx_2 \wedge dy_2$. Notice that the natural symplectomorphism between $K(R)$ and $K'(R)$ has Lipschitz constant $2R$, and in particular, the effect of replacing $K(R)$ with $K'(R)$ only affects our proposition by a factor of $R$ which gets absorbed into the constant $C$. And finally, we allow the Lipschitz constant to be $O_R(L)$, by which we mean it is bounded by $AL$ where $A$ is a constant which again depends upon $R$; we arrive at the proposition as stated by absorbing this constant in $C$.

Consider now each symplectic factor $[0,4R^2] \times [0,1] \subset \R^2$ of $K'(R)$. For $i \in \Z$ with $0 \leq i \leq 4R^2$, let $X_i$ be the region in this rectangle with
$$i \leq x \leq i+1-\frac{1}{L}.$$
Our goal is to fit each $X_i \times X_j \subset K'(R)$ into $Z(\pi )$; indeed the complement of the union of these regions has volume $\Theta(R^4/L)$; the $R^4$ factor gets absorbed by the constant $C$.

To begin, there is an area-preserving map of the rectangle $[0,4R^2] \times [0,1]$ into $\R^2$ of Lipschitz constant at most $O(L)$ which translates $X_i$ in the $x$-direction by $i$. That is, $X_0$ stays fixed, $X_1$ gets shifted to the right by $1$, and the region between them gets stretched like taffy in an area-preserving way to accommodate for this shift. See Figure \ref{fig:separate_horizontal}. Let $Y_i$ be the image of $X_i$ under this map. Then $Y_i$ and $Y_{i+1}$ are separated by distance $1+1/L$, with
$$Y_i = \left[2i,2i+1-\frac{1}{L}\right] \times [0,1].$$

\begin{figure}
	\centering
	\includegraphics[width=\textwidth]{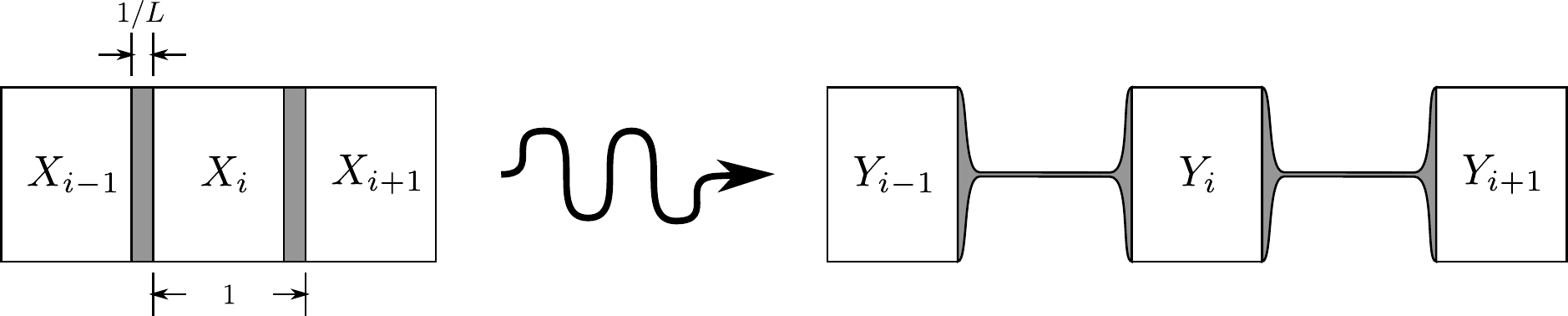}
	\caption{We stretch the region between $X_i$ and $X_{i+1}$ in an area-preserving way so that the images $Y_i$ and $Y_{i+1}$ are separated by distance just over $1$.}
	\label{fig:separate_horizontal}
\end{figure}

Explicitly, a model for the taffy-stretching map is given as $\phi \colon [0,1/L] \times [0,1] \rightarrow [0,1+1/L] \times [0,1]$ of the form
$$\phi(x,y) = \left(f(x), \frac{1}{2} + \frac{y-\frac{1}{2}}{f'(x)}\right)$$
(which is automatically area preserving; see the derivative computed below) with $f \colon [0,1/L] \rightarrow [0,1+1/L]$ a family of functions, depending upon $L$, such that
\begin{itemize}
	\item $f'(x) = 1$ for $x$ in an open neighborhood of the endpoints $0$ and $1/L$
	\item $f(0) = 0$, $f\left(\frac{1}{L}\right) = 1+ \frac{1}{L}$
	\item $1 \leq f'(x) \in O(L)$
	\item $\left|\frac{f''(x)}{(f'(x))^2}\right| \in O(L)$
\end{itemize}
The first two conditions imply that the constructed stretching map $\phi$ glues to the rigid translations of the $X_i$ in a $C^{\infty}$ manner. The latter two conditions imply the desired Lipschitz constant bound of $O(L)$, since
$$D\phi = \begin{pmatrix}f'(x) & 0 \\ -\left(y-\frac{1}{2}\right)\frac{f''(x)}{(f'(x))^2} & \frac{1}{f'(x)}\end{pmatrix},$$
with each entry in this matrix of order $O(L)$.

An example of such a desired function $f$ may be given in the form\footnote{
	There are of course many such choices for function $f$, and we do not claim to make an optimal choice. In an earlier draft, we claimed that the Lipschitz constant of the stretching map could be made less than $2L$. We only need that the Lipschitz constant is $O(L)$ for our purposes, and it is unclear whether $2L$ can be achieved. We thank the reviewer for pointing out this lack of clarity.
}
$$f(x) = \int_0^x \frac{1}{1-Cg(y)}~dy,$$
where $C$ is a constant dependent upon $L$ and $g$ (as we will explain momentarily), and where $g \colon [0,1/L] \rightarrow [0,1/4L]$ is a smooth function which is $C^0$-close to the continuous function
$$g_0(x) = \begin{cases}x & 0 \leq x < 1/4L \\ 1/4L & 1/4L \leq x < 3/4L \\ 1/L - x, & 3/4L \leq x \leq 1/L\end{cases}.$$
We may take $g$ so that $g(x) \equiv 0$ identically near the endpoints $x=0$ and $x=1/L$, and such that $|g'(x)| \leq 1+\epsilon$ for any chosen $\epsilon > 0$. The constant $C$ is chosen so that $f(1/L) = 1+1/L$, i.e. such that
$$I(C) := \int_0^{1/L} \frac{1}{1-Cg(y)}dy = 1+ \frac{1}{L}.$$
We claim that such a constant $C$ exists. Notice that the value of the integral $I(C)$, as a function of $C$, is continuous and monotonically non-decreasing on the interval $[0,1/\sup\{g(x)\})$, with
$$I(0) = \frac{1}{L} < 1+ \frac{1}{L} < \infty = \lim_{C \rightarrow \frac{1}{\sup\{g(x)\}}} I(C),$$
where the limit on the right follows because $g$ attains its maximum value on the interior of the interval $[0,1/L]$, and because $g$ is smooth, we have $g' = 0$ at this maximum value. The existence of $C$ now follows from the intermediate value theorem, and monotonicity implies uniqueness. We notice that because $\sup\{g(x)\} \approx 1/4L$ and $C < 1/\sup\{g(x)\}$, we have that $C \in O(L)$.

With these choices, all of the conditions on $f$ are now met, so long as we take a close enough enough approximation $g$ of $g_0$ (where the closeness depends upon $L$). The first bullet point follows because $g(x) \equiv 0$ near the endpoints. The second follows because we chose $C$ accordingly. The fourth is guaranteed since $$\frac{f''(x)}{(f'(x))^2} = Cg'(x),$$ and we have constructed $g$ so that $|g'(x)| \leq 1+\epsilon$ and $C \in O(L)$.

That leaves the third bullet point, which we now verify. The fact that $f'(x) \geq 1$ is clear, so it suffices to check $f'(x) \in O(L)$. Solving for $C_0$ using $g_0$, we find
\begin{align*}
	1+\frac{1}{L} &= \int_0^{1/L} \frac{1}{1-C_0g_0(x)}dx \\
		&> \int_{1/4L}^{3/4L} \frac{1}{1-C_0/4L}dx \\
		&= \frac{1}{2L} \cdot \frac{1}{1-C_0/4L},
\end{align*}
so that $C_0 < 4L(1- \frac{1}{2(L+1)})$. Notice that the value of $C$ depends continuously on the function $g$ (in the $C^0$-topology), so for any $\epsilon > 0$ there is a choice of approximation $g$ so that
\begin{align*}
	\sup\{f'(x)\} &= \sup \left\{\frac{1}{1-Cg(x)}\right\}\\
		&< \sup\left\{\frac{1}{1-C_0g_0(x)}\right\} + \epsilon\\
		&= \frac{1}{1-C_0/4L} + \epsilon \\
		&< 2(L+1)+\epsilon.
\end{align*}
Hence, $f'(x) \in O(L)$ as required.

Applying the stretching map to each symplectic factor, each region $X_i \times X_j$ is sent to $Y_i \times Y_j$. It suffices now to find a symplectomorphism of $\R^4$ of Lipschitz constant $O_R(1)$ so that each $Y_i \times Y_j$ has image in the cylinder $Z(\pi)$, since in such a case, we compose with our stretching map, and each $X_i \times X_j$ under this composition lands in $Z(\pi )$, where the composition map has total Lipschitz constant $O_R(L)$. We construct this by sliding each $Y_i \times Y_j$ in two steps. We begin by separating the $y_2$-coordinates of the various blocks based on their $x_1$-coordinates. That is, we translate $Y_i \times Y_j$ in the $y_2$-direction by $2i$ units, in other words so that it gets translated to
$$\left[2i,2i+1-\frac{1}{L}\right] \times [0,1] \times \left[2j,2j+1-\frac{1}{L}\right] \times [2i,2i+1].$$
Explicitly, we use a Hamiltonian of the form $H = -\rho(x_1)x_2$ where $\rho$ is a step function with the following properties:
\begin{itemize}
	\item $\rho(x_1) = 2i$ when $2i \leq x_1 \leq 2i+1$, $0 \leq i \leq 4R^2$
	\item $\|\rho'\|_{\infty} \in O(1)$
	\item $\|\rho''\|_{\infty} \in O(1)$
\end{itemize}
The corresponding Hamiltonian vector field $X_H$ is of the form $2i\partial_{y_2}$ when $2i \leq x_1 \leq 2i+1$, and hence translates $Y_i \times Y_j$ as desired. Explicitly, the time-1 Hamiltonian flow is
$$\phi_H^1(x_1,y_1;x_2,y_2) = (x_1,y_1 + \rho'(x_1)x_2; x_2, y_2+\rho(x_1)).$$
with derivative
$$D\phi_H^1 = \begin{pmatrix}1 & 0 & 0 & 0\\ \rho''(x_1)x_2 & 1 & \rho'(x_1) & 0 \\ 0 & 0 & 1 & 0\\ \rho'(x_1) & 0 & 0 & 1 \end{pmatrix}.$$
All terms are $O(1)$ except for $\rho''(x_1)x_2$, because $x_2$ can grow large. But on the image of our cube after the taffy-stretching step, $x_2$ is at most $4R^2$, and so the relevant Lipschitz constant of this sliding step is $O(R^2)$.

A similar construction, using a Hamiltonian of the form $-\rho(y_2)y_1$ for the same function $\rho$, allows us to then take each of these new blocks and translate them in the $x_1$-direction. After we complete both of these steps, the image of $Y_i \times Y_j$ is
$$\left[0,1-\frac{1}{L}\right] \times [0,1] \times \left[2j,2j+1-\frac{1}{L}\right] \times \left[2i,2i+1\right].$$
A final translation simultaneously in the $x_1$ and $y_1$-coordinates by $-\frac{1}{2}$ lands each of these blocks in the cylinder $Z(\pi)$, concluding the proof.
\end{proof}

\begin{remark} \label{rmk:scaling_2} As in Remark \ref{rmk:scaling}, we may also vary the radius of the cylinder $r$, but where the new constant is $r^4C(R/r)$.
\end{remark}

\begin{remark} \label{rmk:constant_2}
The construction as presented has $C(R) \sim R^9$. Indeed, our volume defect came with a factor of $R^4$, but the Lipschitz constant has an extra factor of $R^5$. The first factor of $R$ appearing in the Lipschitz constant came from replacing $K(R)$ with $K'(R)$. An extra two factors of $R^2$ came from our slide moves. One can optimize a little by performing a single diagonal slide move instead of two separate orthogonal slide moves. Hence, in the end, we may take $C(R) \sim R^7$. We suspect this is far from optimal, though decreasing the exponent appears to require a new idea.
\end{remark}

\section{Further questions} \label{sec:questions}

\subsection{Minkowski dimension problem in higher dimensions} \label{ssec:higher_dims}

In this section we pose the simplest Minkowski dimension question that one could ask in dimensions higher than four and make a couple of remarks about it. Let us assume $n>2$ throughout this section.

\begin{question}
What is the smallest $d\in\mathbb{R}$ such that for some $A>\pi$, there exists a closed subset $E\subset B^{2n}(A)$ of Minkowski dimension $d$ so that $B^{2n}(A)\setminus E$ symplectically embeds into $Z^{2n}(\pi)$?
\end{question}

Assume that for some $A>\pi$ and a closed subset $E\subset B^{2n}(A)$ of Minkowski dimension $d$, $B^{2n}(A)\setminus E$ symplectically embeds into $Z^{2n}(\pi).$ We find it plausible that a version of our obstructive argument would still give $d\geq 2$, even though we do not have a proof of this. 

The argument in Proposition \ref{foliation!2} suggests that the problem is related to the question of how the $2$-width of a round ball changes after the removal of a closed subset.  One can explicitly see that for $E$ being the intersection of $B^{2n}(A)$ with a linear Lagrangian subspace $L$ of $\mathbb{R}^{2n}$, an $n$-dimensional submanifold, there is a (holomorphic!) sweepout of $B^{2n}(A)\setminus E$ with width $A/2$. Namely, we take the foliation by half-disks that are the connected components of the intersections with $B^{2n}(A)\setminus E$ of affine complex planes that intersect $L$ non-transversely. 
%
% would show that for any $U$ as in the statement of that proposition, $B^{2n}(A)\setminus U$ admits a sweepout (see \cite{guth} for the relevant definitions) by holomorphic surfaces whose areas are all less than $\pi$. In particular, we obtain that the $2-$width of  $B^{2n}(A)\setminus U$ is less than $\pi$. 
%
%
%This might already give an interesting lower bound on the Minkowski dimension of $E$ given that the $2$-width of $B^{2n}(A)$ {\color{cyan} is {\color{red} bounded below} by $cA$ for some dimensional constant $c$,} by a theorem of Almgren, see page 16 of \cite{guth2}.

\subsection{Capacity after removing a linear plane}\label{sscapquestion}
Throughout this section let $B:=B^4(2\pi)\subset \mathbb{C}^2$, where $\mathbb{C}^2$ is equipped with its standard K\"{a}hler structure. Let us denote the complex coordinates by $x$ and $y$.

Let us denote by $c_{Gr}$ the Gromov width, which is a capacity defined on any symplectic manifold $Y^{2n}$ as the supremum of $\pi r^2$, where $B^{2n}(\pi r^2)$ symplectically embeds into $Y$.

\begin{definition}
Let $V\subset \mathbb{C}^2$ be a real subspace of dimension $2$. We define the symplecticity of $V$ as $|\omega_{st}(e_1,e_2)|$, where $e_1,e_2$ is any orthonormal basis of $T_0V$. 
\end{definition}

Notice that $V$ is a Lagrangian plane if and only if its symplecticity is $0$. On the other extreme, $V$ is a complex plane if and only if its symplecticity is $1$. 

The symplecticity defines a surjective continuous function $$\mathrm{Gr}_{\mathbb{R}}(2,4)\to [0,1].$$
\begin{lemma}
Two elements of $\mathrm{Gr}_\mathbb{R}(2,4)$ have the same symplecticity if and only if there is an element of $U(2)$ sending one to the other.
\end{lemma}

We can therefore define the function (symplecticity to capacity) $$\mathrm{stc}: [0,1]\to (0,\infty), $$ by $s\mapsto c_{Gr}(B\setminus V_s)$, where $V_s\in \mathrm{Gr}_{\mathbb{R}}(2,4)$ with symplecticity $s$.

In Proposition 5.2 of \cite{traynor}, the following remarkable statement is proved.

\begin{proposition}[Traynor]
$B$ is symplectomorphic to $$B\setminus (\{xy=0\}\cup \{(e^{it}x,y)\cup (x, e^{it}y)\mid \Im(x)=\Im(y)=0, \Re(x)\geq 0,\Re(y)\geq 0,t\in [0,2\pi]\}).$$
\end{proposition}

Hence, we have that $\mathrm{stc}(1)=2\pi$. On the other hand, it follows immediately from our Theorem \ref{thmsymplectomorphism} that $\mathrm{stc}(0)=\pi$. We finish with the obvious question.
\begin{question}
What is the function $\mathrm{stc}$? Is it continuous?

\end{question}

\subsection{Minkowski dimension problem for large R} \label{subsec:largeR}
Let $R_\mathrm{sup}\in(1,\infty]$ be the supremum of the radii $R$ such that there is a closed subset $E$ of Minkowski dimension $2$ inside $B(\pi R^2) $ whose complement symplectically embeds into $Z(\pi)$. In Section \ref{sec:construction}, we showed that $R_\mathrm{sup}\leq \sqrt{2}$. This inequality will be improved by Jo\'{e} Brendel to $R_\mathrm{sup}\leq \sqrt{3}$ using a different construction (see Remark \ref{rem-joe}).
An intriguing aspect of both of the squeezing constructions  is that they fail for large radii $R$. This motivates the following:

\begin{question}
Is $R_\mathrm{sup}$ a finite number?
\end{question}

\subsection{Minkowski dimension problem for extendable embeddings}

Here is a variant of our Minkowski dimension question, which is also more directly related to the Lipschitz number question. Assume $R>r$. What is the smallest Minkowski dimension of a subset ${E}\subset B^4(\pi R^2)$ with the property that for any neighborhood $U$ of ${E}$, there is a symplectic embedding $B^4(\pi R^2)\to \mathbb{C}^2$ such that $B^4(\pi R^2)\setminus U$ maps inside $Z^4(\pi r^2)$?

Our obstructive Theorem \ref{thm:minkowski_obstruct} still gives a bound, but our construction in Section \ref{sec:construction} does not apply. Recall $B:=B^4(2\pi)$.

\begin{proposition}
Consider the embedding $B\setminus L\hookrightarrow Z(\pi)$ that we constructed in Theorem \ref{thmsymplectomorphism}. There exists an embedded circle $\eta$ in $iL\cap B$ that is disjoint from $L$, and which maps into $\{x_1=y_1=0\}$.
\end{proposition}

\begin{proof}
Recall that our symplectomorphism first sends $B\setminus L$ to $\mathbb{C}\mathbb{P}^2(2\pi)\setminus (L_{\RP}\cup S_{[0:0:1]})$ via the restriction of a symplectomorphism $b:B\to \mathbb{C}\mathbb{P}^2(2\pi)\setminus S_{[0:0:1]}$. The image of $L\cap B$ under $b$ is $L_{\RP}\setminus S_{[0:0:1]}$, whereas the image of ${iL\cap B}$ is $iL_{\RP}\setminus S_{[0:0:1]}$, where we define $$iL_{\RP}:=\{[z_1:z_2:z_3]\mid \Re(z_1)=\Re(z_2)=\Im(z_3)=0\}.$$ 

$FQ$ and $iL_{\RP}$ intersect along the circle $$\hat{C}=\{[i\cos{\theta}:i\sin{\theta}:1]\mid \theta\in [0,2\pi]\}.$$
Note that $\hat{C}$ is disjoint from $L_{\RP}\cup S_{[0:0:1]}$. 

Now recall that to complete our symplectomorphism from $B\setminus L$ to $\mathcal{E}(\pi,4\pi)\setminus \mathcal{C}$ we use the Oakley-Usher symplectomorphism between $\mathbb{C}\mathbb{P}^2(2\pi)\setminus (L_{\RP}\cup S_{[0:0:1]})$ and  $\overline{D^*\mathbb{RP}^2}\setminus (Z_{\RP}\cup S)$ and then the Opshtein symplectomorphism.

From Opshtein's formula we see that all the points on $Q\setminus (Z_{\RP}\cup S)$ and therefore on $FQ\setminus (L_{\RP}\cup S_{[0:0:1]})$ are sent to points in $\{x_1=y_1=0\}$ in $\mathcal{E}(\pi,4\pi)$. This means that $b^{-1}(\hat{C})$ satisfies the condition in the statement.
\end{proof}

\begin{corollary}\label{cornonextend} Let $U$ be a neighborhood of $L$ that is disjoint from $\eta$. Then, we cannot extend the restriction to $B\setminus U$ of our symplectic embedding $B\setminus L$ into $Z(\pi)$ to a symplectic embedding $B\to\mathbb{C}^2$.
\end{corollary}
\begin{proof}
If there were such an embedding, the action of the image of $\eta$ would have to be simultaneously zero and nonzero, which is absurd. 
\end{proof}

\subsection{Bounds on Minkowski content of the defect region}

We have shown in Corollary \ref{cor:dimension} that the lower Minkowski dimension bound $\underline{\dim}_\mathcal{M} ({E})\geq 2$ is optimal in the range $R\in (1,\sqrt{2}]$. Is the estimate on the 2-content $\underline{\mathcal{M}}_{2} ({E}) \geq \pi(R^2-1)$ also sharp? That is, does there exist $R\in(1,\sqrt{2}]$ and ${E}$ with $\underline{\mathcal{M}}_{2} ({E}) = \pi(R^2-1)$ such that $B(\pi R^2)\setminus {E}$ symplectically embeds into $Z(\pi)$?

\subsection{Speculations on the Lipschitz question}

Consider the volume loss function
$$\mathrm{VL}(L,R) := \inf \left\{\Vol({E}(\Phi))\right\},$$
where the infimum is taken over all symplectic embeddings $\Phi \colon B^{4}(\pi R^2) \rightarrow \R^4$ with Lipschitz constant at most $L$, and 
$${E}(\Phi) = \Phi^{-1}(\R^{4} \setminus  Z^{4}(\pi)).$$
In Section \ref{sec:Lipschitz}, we proved Theorems \ref{thm:Lip_ob} and \ref{thm:Lip_con}, which may be summarized by the statement that
$$\frac{c(R)}{L^2} \leq \mathrm{VL}(L,R) \leq \frac{C(R)}{L}$$
(where we tacitly assume $R > 1$). Even more, we noted in Remarks \ref{rmk:constant} and \ref{rmk:constant_2} that our methods show that we may take $c(R) \sim R^4$ and $C(R) \sim R^7$.

One natural question is whether we can bring the two bounds closer together. In terms of factors of $R$, we suspect that the asymptotics should indeed be $\Theta(R^4)$, i.e. growing like the total volume, though we could not find constructions which remain under this upper bound for large $L$. As for the factors of $L$, the jury is very much out. We nonetheless formulate a precise conjecture.

\begin{conjecture} \label{bb}
For any $\ell,r > 0$, the limit
$$\lim_{R \rightarrow \infty} \lim_{L \rightarrow \infty} \frac{\mathrm{VL}(\ell L, r R)}{\mathrm{VL}(L,R)}$$
exists and is positive.
\end{conjecture}

The second question is to what extent our methods work in higher dimensions. On the constructive side, we may simply take our constructed symplectic embedding in 4 dimensions and extend it to $2n$ dimensions by acting by the identity on the symplectically complementary $2n-4$ dimensions. This yields a constructive bound of $\Theta(R^{2n+3}/L)$. As for the obstructive bound, a modification of the techniques presented in this paper, in which we obtain a sweepout instead of a foliation if we follow Gromov's non-squeezing argument, should probably yield a bound of $O(R^{2n}/L^{2n-2})$. 
%This should be readily tractable, though we have not worked out the details since we would prefer to answer the four-dimensional version completely before we move on to higher dimensions.

Finally, we describe a quantity which we believe could be interesting. Although we are working with balls $B^4(\pi R^2)$, we could in principle replace these with other subdomains in $\R^4$. To be precise, suppose we fix $X \subset \R^4$ a bounded domain. Consider the generalized volume loss function
$$\mathrm{VL}_{X}(L,R) := \inf \{\Vol_4\big({E}(\Phi)\big)\}$$
where the infimum is taken over all symplectic embeddings $\Phi \colon RX \rightarrow \R^4$ of Lipschitz embedding at most $L$, and we take
$${E}(\Phi) := \Phi^{-1}(\R^4 \setminus Z^4(\pi)).$$
In the case $X = B^4(\pi)$, we recover the usual volume loss function above.

We offer the following reasonable-looking conjecture.

\begin{conjecture} \label{cc}
For all bounded domains $X$, the limit
$$s_X := \lim_{R \rightarrow \infty} \lim_{L \rightarrow \infty} \frac{\mathrm{VL}_{X}(L,R)}{\mathrm{VL}(L,R)}$$
exists and is strictly positive.
\end{conjecture}

Notice that if there exists a symplectic embedding $rX \hookrightarrow Y$ of Lipschitz constant $\ell$, then
$$\mathrm{VL}_{X}(\ell L, rR) \leq \mathrm{VL}_Y(L,R).$$
Should Conjecture \ref{cc} hold, one may hope to use this inequality to compare the values of $s_X$ and $s_Y$ for bounded domains $X$ and $Y$.

\appendix
\maketitle

\section{Squeezing and degenerations of the complex projective plane -- by Jo\'{e} Brendel} \label{appx}

\subsection{Introduction and main theorem}
Our goal is to show the following.

\begin{theorem}
	\label{thm:main}
	For every~$\alpha < 3$, there is a set~$\Sigma \subset B^4(\alpha)$ of Minkowski dimension~$2$ such that~$B^4(\alpha)\setminus \Sigma$ symplectically embedds into~$Z^4(1)= \R^2 \times D^2(1) \subset \R^4$.
\end{theorem}

Our notation corresponds to that of the main body of the text by setting~$\alpha = \pi R^2$. The idea of the proof is to view~$B^4(\alpha)$ as~$\CP^2(\alpha) \setminus \CP^1$ and to use almost toric fibrations of~$\CP^2$. As observed in~\cite{Via16} for every \emph{Markov triple}~$(a,b,c)$, there is a triangle~$\Delta_{a,b,c}(\alpha) \subset \R^2$ and an almost toric fibration on~$\CP^2(\alpha)$ with a base diagram whose underlying polytope is~$\Delta_{a,b,c}(\alpha)$. Now note that the toric moment map image of~$Z^4(1)$ is the half-strip~$\cs = \R_{\geqslant 0} \times [0,1)$. We shall show that if the triangle~$\Delta_{a,b,c}(\alpha)$ fits into~$\cs$ (after applying an integral affine transformation), then there is a symplectic embedding of~$\CP^2(\alpha)$ into~$Z^4(1)$ at the cost of removing a certain subset~$\Sigma'$ from~$\CP^2(\alpha)$. The point here is that one can get a good understanding of the subset one needs to remove. Indeed, we show that~$\Sigma'$ is a union of three Lagrangian pinwheels (defined as in \cite{EvaSmi18}) and a symplectic torus. In particular, this set has Minkowski dimension~$2$. A combinatorial argument shows that for every~$\alpha < 3$, there is a Markov triple~$(a,b,c)$ and an inclusion~$\Delta_{a,b,c}(\alpha) \subset \cs$, see Lemma~\ref{lem:markov}.  

\begin{remark}
	{\rm
		As was pointed out to us by Leonid Polterovich, our results can be combined with Gromov's non-squeezing to show that any symplectic ball~$B^4(1+\varepsilon) \subset \CP^2(\alpha)$ intersects the set~$\Sigma' \subset \CP^2$ discussed above. See Corollary~\ref{cor:rigidity} for more details.
	}
\end{remark}

\begin{remark}
	{\rm
		The same strategy may work to produce symplectic embeddings~$D^2(\alpha)\times D^2(\alpha) \setminus \Sigma \hookrightarrow Z^4(1)$ of the polydisk of capacity~$\alpha < 2$ minus a union of some two-dimensional manifolds into the cylinder. Indeed, one can view the polydisk as the affine part of~$S^2 \times S^2$ and use almost toric fibrations of the latter space to carry out the same argument. 
	}
\end{remark}

The relationship between Markov triples and the complex and symplectic geometry of~$\CP^2$ has generated a lot of interest in recent years. It first appeared in the work of Galkin--Usnich~\cite{GalUsn10}, where the authors conjectured that for every Markov triple there is an exotic Lagrangian torus in~$\CP^2$. This conjecture was proved and generalized by Vianna~\cite{Via16},~\cite{Via17} by the use of almost toric fibrations, see also Symington~\cite{Sym03}. On the algebro-geometric side, Hacking--Prokhorov~\cite{HacPro10} showed that a complex surface~$X$ with quotient singularities admits a~$\Q$-Gorenstein smoothing to~$\CP^2$ if and only if~$X$ is a weighted projective space~$\CP(a^2,b^2,c^2)$ and~$(a,b,c)$ forms a Markov triple. In~\cite{EvaSmi18}, Evans--Smith studied embeddings of \emph{Lagrangian pinwheels} into~$\CP^2$. This is directly related to~\cite{HacPro10}, since Lagrangian pinwheels appear naturally as vanishing cycles of the smoothings of~$\CP(a^2,b^2,c^2)$ to~$\CP^2$. See also the recent work by Casals--Vianna~\cite{CasVia20} and the forthcoming paper joint with Mikhalkin and Schlenk~\cite{BreMikSch21} for other applications of almost toric fibrations to symplectic embedding problems.

\emph{Acknowledgements.} We heartily thank Jonny Evans, Grisha Mikhalkin, Leonid Polterovich, Kevin Sackel and Umut Varolgunes for valuable discussions and careful comments on a previous draft. We are grateful to the attentive referee for pointing out a gap in the proof. We thank Felix Schlenk for constant encouragement and generous help. 

\subsection{Some geometry of Markov triangles}
\label{ssec:2}
Let us recall some facts about Markov numbers and their associated triangles.

\begin{definition}
	A triple of natural numbers~$a,b,c \in \N_{>0}$ is called a \emph{Markov triple} if it solves the \emph{Markov equation}
	\begin{equation}
		\label{eq:Markov}
		a^2 + b^2 + c^2 = 3abc.
	\end{equation}
\end{definition}

If~$(a,b,c)$ is a Markov triple, then so is~$(a,b,3ab - c)$. Starting from the solution~$(1,1,1)$, we obtain the so-called \emph{Markov tree} by mutations~$(a,b,c) \rightarrow (a,b,3ab - c)$. The first few Markov triples are
\begin{equation}
	(1,1,1)\quad
	(1,1,2)\quad
	(1,2,5)\quad
	(1,5,13)\quad
	(2,5,29)\quad 
	(1,13,34) \quad \cdots
\end{equation}

Given~$\alpha > 0$, let~$\CP^2$ be equipped with the Fubini--Study symplectic form~$\omega$ normalized such that~$\int_{\CP^1}\omega = \alpha$. For every Markov triple~$(a,b,c)$, there is an almost toric fibration of~$\CP^2$ with almost toric base diagram a rational triangle $\Delta_{a,b,c}(\alpha) \subset \R^2$, which we call the \emph{Markov triangle} associated to the Markov triple~$(a,b,c)$. The first Markov triangle is the (honest) toric moment map image of~$\CP^2$ in our normalization,
\begin{equation}
	\Delta_{1,1,1}(\alpha) = \{ (x,y) \in \R_{\geqslant 0}^2 \, \vert \, x+y \leqslant \alpha \}.
\end{equation}
In fact, we will slightly abuse notation and sometimes think of~$\Delta_{a,b,c}(\alpha)$ as a triangle in~$\R^2$ and sometimes as the equivalence class of triangles~$\R^2$ under the group of toric symmetries given by \emph{integral affine transformations}, i.e.\ elements in~$\operatorname{Aff}(2;\Z) = \R^2 \rtimes \operatorname{GL}(2;\Z)$. For every mutation of Markov triples~$(a,b,c) \rightarrow (a,b,3ab - c)$, there is a corresponding mutation of triangles~$\Delta_{a,b,c}(\alpha) \rightarrow \Delta_{a,b,3ab - c}(\alpha)$, defined by cutting the triangle in two halves and applying a shear map to one of the halves and glueing it back to the other half. This is called a \emph{branch move} and we refer to Symington~\cite[Sections 5.3 and 6]{Sym03} and Vianna~\cite[Section 2]{Via16} for details. For a concrete description of the triangles~$\Delta_{a,b,c}(\alpha)$, see~\eqref{eq:phidelta} and the surrounding discussion. The area of the Markov triangles is well-defined, since it is invariant under~$\operatorname{Aff}(2;\Z)$. Furthermore, the area is invariant under the mutation of triangles and hence we obtain
\begin{equation}
	\label{eq:area}
	\operatorname{area}(\Delta_{a,b,c}(\alpha))
	= \operatorname{area}(\Delta_{1,1,1}(\alpha))
	= \frac{\alpha^2}{2}.
\end{equation}
Recall that we are interested in \emph{embedding Markov triangles into the half-strip}~$\cs = \R_{\geqslant 0} \times [0,1)$. This means that, given~$\alpha > 0$, we look for a Markov triple~$(a,b,c)$ such that~$\Delta_{a,b,c}(\alpha)\subset \cs$ up to applying an element in~$\operatorname{Aff}(2;\Z)$. We will prove the following.

\begin{lemma}
	\label{lem:markov}
	For every~$0 < \alpha < 3$, there is a Markov triple~$(a,b,c)$ such that the Markov triangle~$\Delta_{a,b,c}(\alpha) \subset \R^2$ is in~$\cs = \R_{\geqslant 0} \times [0,1)$ up to applying an element in~$\operatorname{Aff}(2;\Z)$. This result is sharp in the sense that for~$ \alpha \geqslant 3$, there is no such Markov triple.
\end{lemma}

Let us introduce some definitions from integral affine geometry. A vector~$v \in \Z^2$ is called \emph{primitive} if~$\beta v \notin \Z^2$ for all~$0 < \beta < 1$. Note that to every vector~$w \in \R^2$ with rational slope, we can associate a unique primitive vector~$v$ such that~$w = \gamma v$ for~$\gamma > 0$. We call~$\gamma$ the \emph{affine length} of~$w$ and denote it by~$\ell_{\rm aff}(w)$. Let~$l \subset \R^2$ be a rational affine line (or line segment) with primitive directional vector~$v \in \Z^2$ and~$p \in \R^2$ be a point. Then the \emph{affine distance} is defined as~$d_{\rm aff}(p,l) = \vert\operatorname{det}(v,u) \vert$, where~$u$ is any vector such that~$p + u \in l$. This does not depend on any of the choices we have made and these quantities are~$\operatorname{Aff}(2;\Z)$-invariant. See McDuff~\cite{McD11} for more details. 

For a given rational (not necessarily Markov) triangle~$\Delta$, we denote by~$E_1,E_2,E_3$ its edges and by~$v_1,v_2,v_3$ its vertices such that~$v_i$ lies opposite to the edge~$E_i$.
We call~$\ell_{\rm aff}(E_1) + \ell_{\rm aff}(E_2) + \ell_{\rm aff}(E_3)$ the \emph{affine perimeter} of~$\Delta$. Note that the affine perimeter of a Markov triangle~$\Delta_{a,b,c}(\alpha)$ is equal to~$3\alpha$.
We have the following formula for the area of~$\Delta$,
\begin{equation}
	\label{eq:affarea}
	\area(\Delta) = \frac{1}{2} \ell_{\aff}(E_i) \, d_{\aff}(v_i,E_i).
\end{equation}
This follows from the definition of affine distance and affine length. 
\medskip

\ni
\emph{Proof of Lemma~\ref{lem:markov}.}
For the first part of the proof we only need one branch of the Markov tree, namely the one where the maximal entry grows the fastest. More precisely, let~$\Delta_n$ be the Markov triangle associated to the triple~$(m_{n+2},m_{n+1},m_n)$, where~$m_k$ is recursively defined as
\begin{equation}
	\label{eq:mk}
	m_0=m_1=m_2=1, \quad m_{k+2}= 3m_{k+1}m_{k} - m_{k-1}.
\end{equation}
The first few terms of this sequence are given by~$\{m_k\}_{k\in \N}= \{1,1,1,2,5,29,433,\ldots$\}. We clearly have~$m_k \rightarrow \infty$. Recall from~\cite{Via16} that the affine side lengths of a Markov triangle~$\Delta_{a,b,c}(\alpha)$ are given by~$\lambda a^2,\lambda b^2, \lambda c^2$ for a proportionality constant~$\lambda > 0$. Since the affine perimeter of the Markov triangle is~$3\alpha$, we obtain~$\lambda(a^2+b^2+c^2) = 3 \alpha$. Together with the Markov equation~\eqref{eq:Markov} this yields that the longest edge~$E_n$ in~$\Delta_n(\alpha)$ has affine length
\begin{equation}
	\ell_{\rm aff}(E_n) 
	= \frac{\alpha m_{n+2}}{m_{n+1}m_n} 
	\stackrel{\eqref{eq:mk}}{=} \alpha \left( 3 - \frac{m_{n-1}}{m_{n+1}m_n} \right).
\end{equation}
Since the second summand goes to~$0$ for large~$n$, we obtain~$\ell_{\rm aff}(E_n) \rightarrow 3\alpha$. Let~$v_n$ be the vertex opposite to~$E_n$. By~\eqref{eq:area} and the area formula~\eqref{eq:affarea} we obtain
\begin{equation}
	\ell_{\rm aff}(E_n)\,d_{\rm aff}(v_n,E_n)
	= \alpha^2.
\end{equation}
This implies that~$d_{\rm aff}(v_n,E_n) \rightarrow \alpha/3$. Now note that we can assume, up to~$\operatorname{Aff}(2;\Z)$, that the maximal edge~$E_n$ lies in~$\R_{\geqslant 0} \times \{0\}$ and that~$\Delta_n(\alpha)$ lies in the upper half-plane. Then the (Euclidean) height of~$\Delta_n(\alpha)$ is equal to~$d_{\rm aff}(v_n,E_n)$. This implies that for all~$\alpha < 3$, we have~$\Delta_n(\alpha) \subset \cs$ for large enough~$n \in \N$. 

Now let~$\alpha \geqslant 3$ and suppose that there is a Markov triangle~$\Delta_{a,b,c}(\alpha) \subset \cs$. Let~$E$ be the longest edge of~$\Delta_{a,b,c}(\alpha)$ and~$v$ the opposite vertex. We first prove that~$E$ is parallel to~$e_1 = (1,0)$. Indeed, suppose it is not. Then we can write~$E = \ell_{\rm aff}(E) v$ for a primitive vector~$v = (v_1,v_2)\in \Z^2$ with~$v_2 \geqslant 1$. But since the affine perimeter of~$\Delta_{a,b,c}(\alpha)$ is~$3\alpha$ and~$E$ is the longest edge, we obtain~$\ell_{\rm aff}(E) \geqslant \alpha$ and thus~$\Delta_{a,b,c}(\alpha)$ is not contained in~$\cs$. Now if~$E$ is parallel to~$e_1$, then the (Euclidean) height of~$\Delta_{a,b,c}(\alpha)$ is~$d_{\rm aff}(v,E)$. The affine perimeter~$3\alpha$ is strictly larger than~$\ell_{\rm aff}(E)$ from which we deduce~$d_{\rm aff}(v,E) >  \alpha/3 \geqslant 1$ by the area formula~\eqref{eq:affarea}.
\proofend

\subsection{Proof of the main theorem}

\begin{definition}
	Let~$p$ be a positive integer. The topological space obtained from the unit disk~$D$ by quotienting out the action of the group of the~$p$-th roots of unity on~$\pp D$ is called $p$-\emph{pinwheel}.
\end{definition}

For example the~$2$-pinwheel is~$\RP^2$. The image of~$\pp D$ in the quotient is called the \emph{core circle}. For all~$p > 2$ the~$p$-pinwheels are not smooth at points of the core circle. A \emph{Lagrangian pinwheel} in a symplectic manifold~$M$ is a Lagrangian \emph{embedding} of a~$p$-pinwheel into~$M$, see~\cite[Definition 2.3]{EvaSmi18} for the meaning of \emph{embedding} in this context. As it turns out, for every Lagrangian $p$-pinwheel, there is an additional extrinsic parameter~$q \in \{1,\ldots,p-1\}$ measuring the twisting of the pinwheel around its core circle. We call such an object~$(p,q)$-pinwheel and denote it by~$L_{p,q}$ when this causes no confusion. For us, the following result is key.

\begin{proposition}
	\label{lem:main}
	Let~$(a,b,c)$ be a Markov triple and let~$s_a,s_b,s_c \in \CP(a^2,b^2,c^2)$ be the orbifold singular points of the corresponding weighted projective space. Then there is a surjective map~$\tilde{\phi} \colon \CP^2 \rightarrow \CP(a^2,b^2,c^2)$ and there are (mutually disjoint) Lagrangian pinwheels~$L_{a,q_a}, L_{b,q_b}, L_{c,q_c} \subset \CP^2$ with
	\begin{equation}
		\tilde{\phi}(L_{a,q_a}) = s_a, \quad
		\tilde{\phi}(L_{b,q_b}) = s_b, \quad
		\tilde{\phi}(L_{c,q_c}) = s_c,
	\end{equation}
	such that~$\tilde{\phi}$ restricts to a symplectomorphism
	\begin{equation}
		\label{eq:mainsymplecto}
		\phi \colon  \CP^2\setminus (L_{a,q_a} \sqcup L_{b,q_b} \sqcup L_{c,q_c}) \rightarrow \CP(a^2,b^2,c^2) \setminus \{s_a,s_b,s_c\}.
	\end{equation}
	Furthermore, the preimage of the set~$\cd \subset \CP(a^2,b^2,c^2)$, see~\eqref{eq:ellbdry}, consists of the union of three Lagrangian pinwheels and a symplectic two-torus intersecting the pinwheels in their respective core circles.
\end{proposition}

\begin{remark} 
	{\rm
		For the construction of this symplectomorphism, it seems plausible that one can use the existence of a~$\Q$-Gorenstein smoothing of~$\CP(a^2,b^2,c^2)$ to~$\CP^2$ with vanishing locus consisting of a union of three pinwheels. Such smoothings were constructed in~\cite{HacPro10}, by showing that there are no obstructions to piecing together local smoothings of the cyclic quotient singularities holomorphically. Our proof follows the same strategy in the symplectic set-up, see also~\cite[Examples 2.5/2.6]{EvaSmi18} for a discussion of this, ~\cite[Section 3]{LekMay14} and~\cite[Section 1]{Eva19} for the local smoothings. 
	}
\end{remark}

We now turn to the proof of Theorem~\ref{thm:main} using Proposition~\ref{lem:main}, the proof of which we postpone to~\S\ref{ssec:4}.\smallskip

\ni
\emph{Proof of Theorem~\ref{thm:main}.}
\emph{Step 1:} Let~$\alpha < 3$ and choose~$\alpha'$ such that~$\alpha < \alpha' < 3$. Pick a Markov triple~$(a,b,c)$ and an associated Markov triangle~$\Delta_{a,b,c} = \Delta_{a,b,c}(\alpha')$ which lies in~$\cs= \R_{\geqslant 0} \times [0,1)$. This is possible by Lemma~\ref{lem:markov}. Note that~$\Delta_{a,b,c}$ is the image of the toric orbifold moment map~$\mu \colon \CP(a^2,b^2,c^2) \rightarrow \Delta_{a,b,c}$, provided we normalize the orbifold symplectic form appropriately. See the discussion surrounding~\eqref{eq:mu} for details on the toric structure on weighted projective space. Let
\begin{equation}
	\label{eq:ellbdry}
	\cd = \mu^{-1}(\pp \Delta_{a,b,c}) \subset \CP(a^2,b^2,c^2)
\end{equation}
be the preimage of the boundary. The set~$\cd$ is a union of complex suborbifolds~$\CP(a^2,b^2) \cup \CP(b^2,c^2) \cup \CP(a^2,c^2)$ such that each of these suborbifolds projects to one edge of the triangle~$\Delta_{a,b,c}$. The complement of~$\cd$ admits a symplectic embedding,
\begin{equation}
	\label{eq:emb1}
	\psi\colon	
	\CP(a^2,b^2,c^2) \setminus \cd
	\hookrightarrow
	Z^4(1).
\end{equation}
Indeed, this follows from the inclusion~$\operatorname{int}(\Delta_{a,b,c}) \subset \operatorname{int}(\cs)$ and the fact that inclusions of toric moment map images which respect the boundary stratifications yield (equivariant) symplectic embeddings, see for example~\cite{Sym03}.\smallskip

\emph{Step 2:} By Proposition~\ref{lem:main} there is a symplectomorphism~$\phi$ from the complement of Lagrangian pinwheels~$L_{a,q_a}, L_{b,q_b}, L_{c,q_c}$ to the complement of the orbifold points of~$\CP(a^2,b^2,c^2)$. By restricting~$\phi$, we obtain the symplectomorphism,
\begin{equation}
	\label{eq:emb2}
	\phi' \colon
	\CP^2\setminus \Sigma'
	\rightarrow	
	\CP(a^2,b^2,c^2)\setminus \cd.
\end{equation}
Again by Proposition~\ref{lem:main}, the set~$\Sigma'$ consists of the union of three pinwheels and a symplectic two-torus. 
\smallskip

\emph{Step 3:} The standard embedding~$B^4(\alpha) \subset \CP^2(\alpha')$ together with~$\phi'$ and~$\psi$ yields an embedding~$B^4(\alpha) \setminus \Sigma \hookrightarrow Z^4(1)$. Here~$\Sigma$ denotes~$\Sigma = \Sigma' \cap  B^4(\alpha)$. The set~$\Sigma$ has Minkowski dimension two. Indeed, the embedding~$B^4(\alpha) \subset \CP^2(\alpha')$ is bilipschitz (its image being contained in a closed ball) and volume preserving and the set~$\Sigma'$ consists of the union of three pinwheels and a symplectic two-torus.
\proofend

As was pointed out to us by Leonid Polterovich, one can combine Theorem~\ref{thm:main} with Gromov's non-squeezing theorem to get certain rigidity results, reminiscent of~\cite[Theorem 1.B]{biran}.  

\begin{corollary}
	\label{cor:rigidity}
	Let~$\Sigma' \subset \CP^2(\alpha)$ be one of the above sets such that~$\CP^2(\alpha)\setminus \Sigma'$ embeds into~$Z^4(1)$ for some~$1< \alpha < 3$. Then every symplectic ball~$B^4(1+\varepsilon) \subset \CP^2$ for~$\varepsilon > 0$ intersects~$\Sigma'$.
\end{corollary}

\proof
Assume~$B^4(1+\varepsilon) \subset \CP^2(\alpha)$ does not intersect~$\Sigma'$. The embedding~$\CP^2(\alpha)\setminus \Sigma' \hookrightarrow Z^4(1)$ yields a symplectic embedding~$B^4(1+\varepsilon) \hookrightarrow Z^4(1)$, contradicting non-squeezing.
\proofend

Note that for a fixed~$1 < \alpha < 3$ we get infinitely many sets~$\Sigma' \subset \CP^2$ to which Corollary~\ref{cor:rigidity} applies and all of these consist of a union of a symplectic torus and Lagrangian pinwheels.

\subsection{Proof of Proposition~\ref{lem:main}}
\label{ssec:4}

Following the exposition in~\cite[Example 2.5]{EvaSmi18} we consider smoothings of certain orbifold quotients of~$\C^2$. This yields the local version from Lemma~\ref{lem:local} of the symplectomorphism in Proposition~\ref{lem:main}. 

Let~$a,q$ be coprime integers with~$1\leqslant q < a$ and take the quotient of~$\C^2$ by the following action of~$a^2$-th roots of unity,
\begin{equation}
	\label{eq:action1}
	\zeta . (z_1,z_2) = (\zeta z_1 , \zeta^{aq -1}z_2), \quad
	\zeta^{a^2}=1.
\end{equation}
We denote this quotient by~$\C^2/\Gamma_{a,q}$. It can be embedded as~$\{w_1w_2 = w_3^a\}$ into the quotient~$\C^3/\Z_a$ by the action
\begin{equation}
	\eta . (w_1,w_2,w_3) = (\eta w_1, \eta^{-1} w_2, \eta^q w_3), \quad
	\eta^a = 1.
\end{equation}
The smoothing is given by 
\begin{equation}
	\label{eq:chi}
	\cx = \{ w_1w_2 = w_3^a +t \} \subset \C^3/\Z_a \times \C_t,
\end{equation}
which we view as a degeneration by projecting to the~$t$-component,~$\pi \colon \cx \rightarrow \C_t$. We denote the fibres by~$X_t = \pi^{-1}(t)$. The smooth fibre~$X_1$ is a rational homology ball and the vanishing cycle of the degeneration is a Lagrangian pinwheel~$L_{a,q}$. This follows from the description of~$\cx$ as~$\Z_a$-quotient of an~$A_{a-1}$-Milnor fibre. Let~$s \in X_0$ be the unique isolated singularity of~$X_0$, and~$\cx^{\rm reg} = \cx \setminus \{s\}$ its complement. The restriction of the standard symplectic form~$\omega_0$ on~$\C^4 = \C^3 \times \C_t$ yields a symplectic manifold~$(\cx^{\rm reg}, \Omega)$. Note that the smooth loci of the fibres~$X_{t \neq 0}$ and~$X_0 \setminus \{s\}$ are symplectic submanifolds. Let us now construct a symplectomorphism 
\begin{equation}
	\label{eq:sigma}
	\psi \colon X_1 \setminus L_{a,q} \rightarrow X_0\setminus \{s\} = \C^2/\Gamma_{a,q} \setminus \{0\}.
\end{equation}
For this, we take the connection on~$\cx^{\rm reg}$ defined as the symplectic complement to the vertical distribution,
\begin{equation}
	\label{eq:sympconn}
	\xi_x 
	= (\ker (\pi_*)_x)^{\Omega}
	= \left\{ v \in T_x \cx^{\reg} \mid \Omega (v,w) = 0 
	\mbox{ for all } w \in T_x \pi^{-1}(\pi(x)) 
	\right\}.
\end{equation}
This connection is symplectic in the sense that its parallel transport maps are symplectomorphisms whenever they are defined. In particular, we get a symplectomorphism between any two regular fibres~$X_{t_1}$ and~$X_{t_2}$ by picking a curve in~$\C^{\times}$ with endpoints~$t_1$ and~$t_2$. Since we are interested in the singular fibre for the construction of~\eqref{eq:sigma}, take the curve~$\gamma(r) = 1-r \in \C$. For every~$r<1$, this yields a symplectomorphism~$\psi^r \colon X_1 \rightarrow X_{1-r}$. As it turns out, setting~$\psi(x) = \lim_{r \rightarrow 1}\psi^r(x)$ for~$x\in X_1$ yields a well-defined surjective map~$\tilde{\psi} \colon X_1 \rightarrow X_0$ with vanishing cycle a Lagrangian pinwheel,~$L_{a,q} = \tilde{\psi}^{-1}(s)$, and which restricts to the desired symplectomorphism~\eqref{eq:sigma}. For the fact that~$\tilde{\psi}$ is well-defined, see the unpublished notes by Evans~\cite[Lemma 1.2]{Eva19}. We refer to~\cite[Section 3.1]{LekMay14} for more details on the specific degeneration we consider above and to~\cite[Lemma 1.20]{Eva19} for the fact that the vanishing cycle is a Lagrangian pinwheel. 

\begin{lemma}
	\label{lem:local}
	\emph{(Evans \cite{Eva19})}
	Let~$\C^2/\Gamma_{a,q}$ the quotient by the action~\eqref{eq:action1} and~$S_{a,q}$ its smoothing as above. Then there is a surjective map~$\tilde{\psi}_a \colon S_{a,q} \rightarrow \C^2/\Gamma_{a,q}$ with~$\tilde{\psi}_a(L_{a,q}) = \{0\}$ and which restricts to a symplectomorphism
	\begin{equation}
		\psi_a \colon S_{a,q} \setminus L_{a,q} \rightarrow (\C^2/\Gamma_{a,q}) \setminus \{0\}.
	\end{equation}
	Furthermore, the preimage of the set~$\cd_{a,q_a} = \{z_1z_2 = 0\}/\Gamma_{a,q}$ under~$\tilde{\psi}_a$ is the union of~$L_{a,q}$ and a symplectic cylinder intersecting~$L_{a,q}$ in its core circle. 
\end{lemma}

\proof
As explained above, the main statement of the lemma follows from~\cite{Eva19}. We only need to identify the preimage of~$\cd_{a,q_a} = \{z_1z_2 = 0\}/\Gamma_{a,q}$, which can be done by keeping track of the parallel transport in the explicit model~\eqref{eq:chi}. Under the identification of~$\C^2/\Gamma_{a,q}$ with~$X_0$ (which we will tacitly use throughout), the set~$\{z_1z_2 = 0\}/\Gamma_{a,q}$ corresponds to~$\{w_3 = 0\}/\Z_a \cap X_0$. We claim that the set~$\bigcup_t (\{w_3 = 0\}/\Z_a \cap X_t) \subset \cx$ is invariant under the symplectic parallel transport induced by~\eqref{eq:sympconn}. This proves that~$\tilde{\psi}_a^{-1}(\{w_3 = 0\}/\Z_a \cap X_0) = (\{w_3 = 0\}/\Z_a \cap X_1) \cup L_{a,q}$. Indeed, the preimage of~$0 \in \C^2/\Gamma_{a,q}$ is~$L_{a,q}$ and on the complement of~$L_{a,q}$, the map~$\tilde{\psi}_a$ is a diffeomorphism. To see that~$\{w_3 = 0\}/\Z_a \cap X_1$ is a symplectic cylinder intersecting~$L_{a,q}$ in its core circle, one can consider the singular fibration structure of the map~$X_s \rightarrow \C, (w_1,w_2,w_3) \mapsto w_3^p$ see~\cite[\S 1.2.3]{Eva19} for more details. Thus it remains to show the invariance of~$\bigcup_t (\{w_3 = 0\}/\Z_a \cap X_t)$. 

Let~$x \in \{w_3 = 0\}/\Z_a \cap X_t$, meaning that~$x$ is a~$\Z_a$-class of a point~$(x_1,x_2,0) \in \C^3$ with~$x_1x_2 = t$. This gives a natural inclusion~$T_{(x,t)}\cx = T_xX_t \oplus T_t\C \subset \C^3 \oplus \C$. Since~$\Omega = \omega_{\C^3} \oplus \omega_{\C}$, the horizontal lift by the symplectic connection~\eqref{eq:sympconn} of a vector~$v \in T_t\C = \C$ is given by~$u + v \in \xi_{(x,t)}$, where~$u \in (T_xX_t)^{\omega_{\C^3}}$. Note that the subspace~$\{u_1 = u_2 = 0\} \subset T_x\C^3 = \C^3$ is contained in~$T_xX_t$ and hence its symplectic complement~$\{u_3 = 0\} \subset \C^3$ contains the symplectic complement~$(T_xX_t)^{\omega_{\C^3}}$. Since~$u$ is contained in~$(T_xX_t)^{\omega_{\C^3}}$, this proves that~$u$ is tangent to the subset~$\{w_3 = 0\}$ and hence this subset is preserved under parallel transport.
\proofend

We turn to the proof of Proposition~\ref{lem:main}. The main idea is to use the fact that~$\Delta_{a,b,c}(\alpha)$ is both the almost toric base polytope of~$\CP^2(\alpha)$ and the toric base of~$\CP(a^2,b^2,c^2)$ with a suitably normalized symplectic form. This shows that there is a symplectomorphism which intertwines the (almost) toric structures on the preimages of a complement of neighbourhoods of the vertices. For example one can choose~$W \subset \Delta_{a,b,c}$ as in Figure~\ref{fig:2}. We use Lemma~\ref{lem:local} to extend this symplectomorphism. 

Let us make a few preparations. In particular, we discuss how to use the quotient~$\C^2/\Gamma_{a,q}$ as a local toric model for~$\CP(a^2,b^2,c^2)$. This is the orbifold version of the toric ball embedding into~$\CP^2$ one obtains by the inclusion of the simplex with one edge removed into the standard simplex~$\Delta_{1,1,1}$. The toric structure on~$\C^2/\Gamma_{a,q}$ is induced by the standard toric structure~$(z_1,z_2)\mapsto(\pi\vert z_1 \vert^2,\pi\vert z_2\vert^2)$ on~$\C^2$. Indeed, note that the~$\Gamma_{a,q}$-action is obtained by restricting the standard~$T^2=(\R/\Z)^2$-action to a discrete subgroup~$\Z_{a^2}$. This implies that we obtain an induced action by~$T^2/\Gamma_{a,q}$ on~$\C^2/\Gamma_{a,q}$. This action is Hamiltonian and its moment map image, under a suitable identification~$T^2 \cong T^2/\Gamma_{a,q}$, is given by
\begin{equation}
	\angle_{a,q} = \{ x_1 w_1 + x_2 w_2 \,\vert\, x_1,x_2 \geqslant 0 \}, \quad 
	w_1 =
	\begin{pmatrix}
		1 \\
		0
	\end{pmatrix},
	w_2 = 
	\begin{pmatrix}
		aq -1 \\
		a^2
	\end{pmatrix}.
\end{equation}
See for example~\cite[Remark 2.7]{EvaSmi18} or~\cite[Section 9]{Sym03}. Note that a ball~$B^4(d) = \{ \pi(\vert z_1 \vert^2 + \vert z_2 \vert^2) < d \} \subset \C^2$ quotients to an orbifold ball~$B^4(d)/\Gamma_{a,q} \subset \C^2/\Gamma_{a,q}$ which is fibred by the induced toric structure on the quotient. Furthermore, the boundary sphere~$S^3(d) \subset \C^2$ quotients to a lens space~$\Sigma_{a,q}(d) = S^3(d)/\Gamma_{a,q}$ of type~$(aq-1,a^2)$ equipped with its canonical contact structure and which fibers over a segment in~$\angle_{a,q}$. We will use this fact in the proof of Proposition~\ref{lem:main}. 

Let us now show that, for a suitable choice of~$q$, the toric system~$\C^2/\Gamma_{a,q} \rightarrow \angle_{a,q}$ can be used as a local model around one of the orbifold points of the toric system on~$\CP(a^2,b^2,c^2)$. In order to get a concrete description of this toric system on the weighted projective space, recall that the symplectic orbifold~$\CP(a^2,b^2,c^2)$ can be defined as a symplectic quotient of~$\C^3$,
\begin{equation}
	\label{eq:habc}
	\CP(a^2,b^2,c^2) = H^{-1}(a^2b^2c^2)/S^1, \quad
	H = a^2 \vert z_1 \vert^2 + b^2 \vert z_2 \vert^2 + c^2 \vert z_3 \vert^2.
\end{equation}
This description~\eqref{eq:habc} has the advantage that it is naturally equipped with a Hamiltonian~$T^2$-action inherited from the standard~$T^3$-action on~$\C^3$. This induced action is toric and its moment map image is given by the intersection of the plane defined by~$H$ and the positive orthant,
\begin{equation}
	\label{eq:deltaabc}
	\tilde{\Delta}_{a,b,c} 
	=
	\{ a^2y_1 + b^2y_2 + c^2y_3 = a^2b^2c^2 \} \cap \R^3_{\geqslant 0}.
\end{equation}
Note that this is a polytope in~$\R^3$ and not~$\R^2$. We get a Markov triangle~$\Delta_{a,b,c}(abc) \subset \R^2$ as in \S\ref{ssec:2} by setting
\begin{equation}
	\label{eq:phidelta}
	\Delta_{a,b,c}(abc) = \Phi^{-1} ( \tilde{\Delta}_{a,b,c}),
\end{equation}
for~$\Phi$ an integral affine embedding (see Definition~\ref{def:iae}) containing~$\tilde{\Delta}_{a,b,c}$ in its image. Recall that this produces the same triangles (up to integral affine equivalence) as those obtained from almost toric fibrations of~$\CP^2$ as discussed in~\S\ref{ssec:2}, see~\cite[Section 2]{Via16}. Hence it makes sense to denote them by~$\Delta_{a,b,c}(abc)$. The normalization~$\alpha = abc$ of the triangle comes from the choice of level at which we have reduced in~\eqref{eq:habc}.

\begin{definition}
	\label{def:iae}
	An affine map~$\Phi \colon \R^2 \rightarrow \R^3, x \mapsto Ax + b$ for~$A \in \Z^{3 \times 2}$ and~$b \in \R^3$ is called \emph{integral affine embedding} if it is injective and if~$A(\Z^2) = A(\R^2) \cap \Z^3$. 
\end{definition}

Note that by the definition of integral affine embedding, the definition~\eqref{eq:phidelta} makes sense and the polytope it defines is independent of the choice of~$\Phi$ up to applying an integral affine transformation. We now show that there is a natural number~$q$ and an integral affine embedding~$\Phi_{a,q}$ such that the triangle~$\Phi^{-1}_{a,q}( \tilde{\Delta}_{a,b,c}) \subset \R^2$ is obtained by intersecting~$\angle_{a,q}$ with a half-plane. Indeed, set
\begin{equation}
	\Phi_{a,q} \colon \R^2 \rightarrow \R^3, \quad
	\begin{pmatrix}
		x_1 \\
		x_2
	\end{pmatrix}
	\mapsto 
	\begin{pmatrix}
		-b^2 & 1 + \frac{b}{a}(bq -3c) \\
		a^2 & 1 - aq \\
		0 & 1
	\end{pmatrix}
	\begin{pmatrix}
		x_1 \\
		x_2
	\end{pmatrix}
	+
	\begin{pmatrix}
		b^2c^2 \\
		0 \\
		0
	\end{pmatrix},
\end{equation}
where~$q$ satisfies~$bq =  3c \; \operatorname{mod} a$, see also~\cite[Example 2.6]{EvaSmi18}. The map~$\Phi_{a,q}$ has image~$\{ a^2y_1 + b^2y_2 + c^2y_3 = a^2b^2c^2 \}$ and it is an integral affine embedding, as can be checked by a computation. Furthermore~$\Phi$ maps~$0$ to the vertex~$(b^2c^2,0,0)$ (corresponding to~$a^2$) of~$\tilde{\Delta}_{a,b,c}$ and~$v_1$ and~$v_2$ to the outgoing edges at~$(b^2c^2,0,0)$. This means that there is an integral vector~$(\xi_1,\xi_2)$ defining a half-plane~$K=\{ \xi_1x_1 + \xi_2x_2 \leqslant k \}$ such that 
\begin{equation}
	\Phi^{-1}_{a,q}( \tilde{\Delta}_{a,b,c})
	=
	K \cap \angle_{a,q}.
\end{equation}
From this we deduce the desired toric model. Let~$\tilde{E}_a \subset \tilde{\Delta}_{a,b,c}$ be the edge opposite the vertex~$(b^2c^2,0,0)$.

\begin{proposition}
	\label{prop:toricmodel}
	The subset in~$\CP(a^2,b^2,c^2)$ fibering over~$\tilde{\Delta}_{a,b,c} \setminus \tilde{E}_a$ is fibred (orbifold-) symplectomorphic to the subset in~$\C^2/\Gamma_{a,q}$ fibering over~$\Int K \cap \angle_{a,q}$.
\end{proposition} 
\proof
We have shown above that~$\tilde{\Delta}_{a,b,c} \setminus \tilde{E}_a$ and~$\Int K \cap \angle_{a,q}$ are integral affine equivalent. This implies the claim by the classification of compact toric orbifolds by their moment map images, see~\cite{LerTol97}. Compactness is not a problem here, since we can compactify the subset fibering over~$K \cap \angle_{a,q}$ by performing a symplectic cut at~$\{\xi_1x_1 + \xi_2x_2 = k\}$. 
\proofend

Let us now fix a moment map
\begin{equation}
	\label{eq:mu}
	\mu \colon \CP(a^2,b^2,c^2) \rightarrow \Delta_{a,b,c} = \Delta_{a,b,c}(abc) \subset \R^2,
\end{equation}
by composing the moment map~$\CP(a^2,b^2,c^2) \rightarrow \tilde{\Delta}_{a,b,c}$ with the inverse of a suitable integral affine embedding as described above. Until the end of the proof of Proposition~\ref{lem:main} we simplify notation by writing~$\Delta_{a,b,c} = \Delta_{a,b,c}(abc)$.\smallskip

\ni
\emph{Proof of Proposition~\ref{lem:main}.} 
The main part of the proof will be concerned with proving the existence of the symplectomorphism~\eqref{eq:mainsymplecto} and for readability, we postpone the proof of the existence of the global map~$\tilde{\phi}$ and the computation of~$\tilde{\phi}^{-1}(\cd)$ to \emph{Step 5}.

\emph{Step 1.} We start by setting up some notation on the side of the weighted projective space. The orbifold points~$s_a,s_b,s_c \in \CP(a^2,b^2,c^2)$ are mapped to the vertices~$v_a,v_b,v_c \in \Delta_{a,b,c}$ under the moment map~$\mu \colon \CP(a^2,b^2,c^2) \mapsto \Delta_{a,b,c}$. Let us first focus on the orbifold point~$s_a$. Denote edge opposite to~$v_a$ by~$E_a$. By Proposition~\ref{prop:toricmodel}, there is an orbifold symplectomorphism
\begin{equation}
	\rho_a \colon \mu^{-1}(\Delta_{a,b,c}\setminus E_a) \rightarrow \mu_{\C^2/\Gamma_{a,q}}^{-1}(\Int K \cap \angle_{a,q})
\end{equation}
which intertwines the toric structures. Now let~$\overline{B^4}(d)/\Gamma_{a,q} \subset \C^2/\Gamma_{a,q}$ be a closed orbifold ball for~$\overline{B^4}(d)=\{\pi(\vert z_1 \vert^2 + \vert z_2 \vert^2) \leqslant d \}$ and~$d>0$. Its boundary~$S^3(d)/\Gamma_{a,q} \subset \C^2/\Gamma_{a,q}$ is a lens space equipped with the standard contact structure. Note that both the orbifold ball and its boundary are fibred by the moment map~$\mu_{\C^2/\Gamma_{a,q}}$. Since~$\rho_a$ intertwines the toric structures, the image sets
\begin{equation}
	\label{eq:orbiballs}
	B^{\rm orb}_{a} = \rho^{-1}_a(\overline{B^4}(d)/\Gamma_{a,q}), \quad
	\Sigma_{a} = \rho^{-1}_a(S^3(d)/\Gamma_{a,q}),
\end{equation}
are fibred by~$\mu$. Then the image of the pair~$(B_{a},\Sigma_{a})$ under~$\mu$ is a pair~$(V_a,\ell_a)$ consisting of a segment contained in a triangle around the vertex~$v_a$. Note also that the lens space~$\Sigma_a$ is naturally equipped with its standard contact structure. We do the same procedure around the remaining vertices~$v_b,v_c$ and denote the corresponding sets by~$B^{\rm orb}_b,B^{\rm orb}_c,\Sigma_b,\Sigma_c \subset \CP(a^2,b^2,c^2)$ and by~$V_b,V_c,\ell_b,\ell_c \subset \Delta_{a,b,c}$. We choose the sizes such that~$B^{\rm orb}_a,B^{\rm orb}_b,B^{\rm orb}_c$ are mutually disjoint. Furthermore, we choose a set~$W \subset \Delta_{a,b,c}$ such that~$\Delta_{a,b,c} = W \cup V_a \cup V_b \cup V_c$ and such that the overlap~$W \cap V_j$ is a strip around~$\ell_j$ for all~$j \in \{a,b,c\}$. Again, see Figure~\ref{fig:2}.

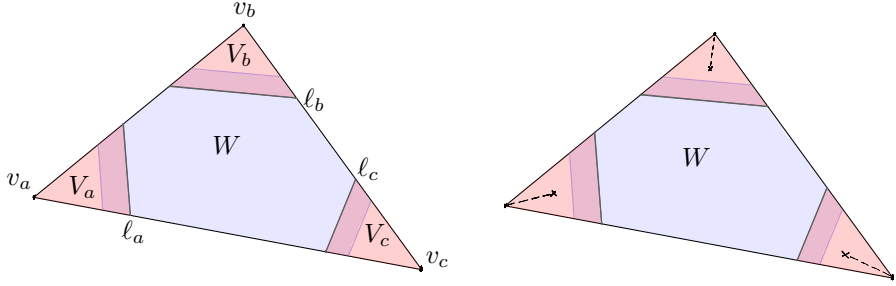
\begin{figure}
	\begin{center}
		\begin{tikzpicture}[scale=0.8]
			\node[inner sep=0pt] at (0,0)
			{\includegraphics[trim={0.5cm 5cm 1.2cm 0.5cm},clip,scale=0.64]						{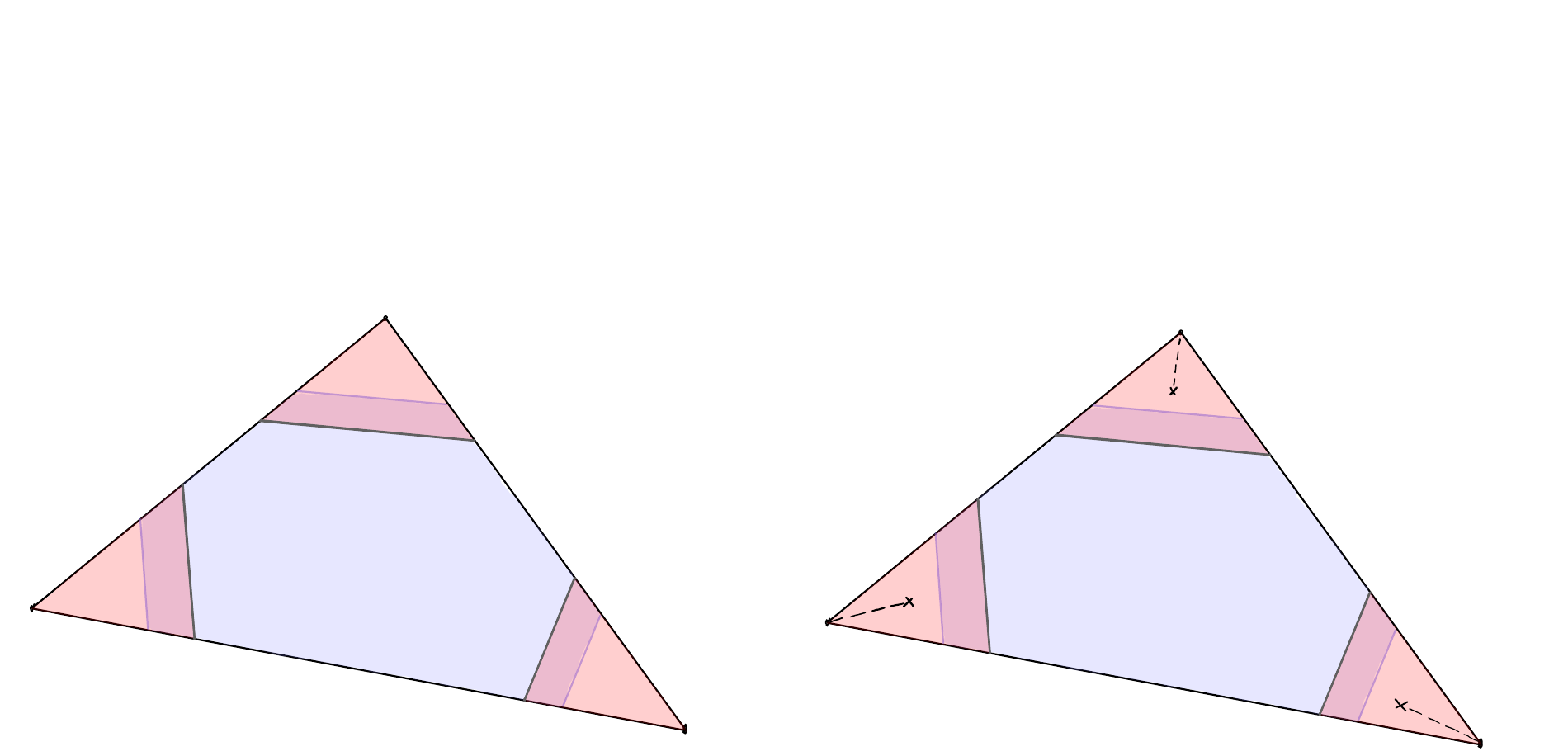}};
			\node at (-4.2,-1.4){$W$};
			\node at (3.6,-1.6){$W$};
			\node at (-7.65,-2){$v_a$};
			\node at (-6.6,-2.1){$V_a$};
			\node at (-5.75,-2.85){$\ell_a$};
			\node at (-3.9,0.85){$v_b$};
			\node at (-4,0.1){$V_b$};
			\node at (-2.75,-0.65){$\ell_b$};
			\node at (-0.7,-3.3){$v_c$};
			\node at (-1.7,-2.9){$V_c$};
			\node at (-1.85,-1.8){$\ell_c$};
		\end{tikzpicture}
		\caption{The triangle~$\Delta_{a,b,c}$ as union~$W \cup V_a \cup V_b \cup V_c$, on the left as the toric moment polytope of~$\CP(a^2,b^2,c^2)$ and on the right as almost-toric base diagram of~$\CP^2$. In both cases the fibration is toric over~$W$ and lens spaces fibre over the segments~$\ell_a,\ell_b,\ell_c$.} 
		\label{fig:2}
	\end{center}
\end{figure}

\emph{Step 2: } Now consider the almost-toric fibration of~$\CP^2$ associated to the triangle~$\Delta_{a,b,c}$. In the conventions of~\cite{Sym03} and~\cite{Via16} the triangle~$\Delta_{a,b,c}$ is decorated with three dashed line segments of prescribed slope between the vertices and the nodal points. The latter are usually marked by a cross. There is a map~$\pi \colon \CP^2 \rightarrow \Delta_{a,b,c}$ which is a standard toric fibration away from the dashed lines, but which is only continuous on the pre-images of the dashed lines (which encode monodromy of the integral affine structure). By applying \emph{nodal slides} if necessary, we may assume that the dashed lines lie outside of the subset~$W \subset \Delta_{a,b,c}$. Since the projection~$\pi$ is standard toric away from the dashed lines, this implies that there is a symplectomorphism
\begin{equation}
	\phi_0 \colon \CP^2 \supset \pi^{-1}(W) \rightarrow \mu^{-1}(W) \subset \CP(a^2,b^2,c^2),
\end{equation}
which intertwines~$\pi$ and~$\mu$. Define the preimages
\begin{equation}
	B'_{a} = \pi^{-1}(V_a), \quad
	\Sigma'_{a} = \pi^{-1}(\ell_a).
\end{equation}
By~\cite[Section 9]{Sym03}, the set~$B_a'$ is a closed rational homology ball and~$\Sigma'_a$ is a lens space of type~$(aq-1,a^2)$ equipped with its standard contact structure. In fact,~$\phi_0$ maps the copy~$\Sigma_a'$ of the lens space to the copy~$\Sigma_a$. However, contrary to~$B_a^{\rm orb}$, the rational homology ball~$B_a'$ is smooth. The same discussion holds for~$b$ and~$c$.

\emph{Step 3: } The key part of the proof is finding extensions
\begin{equation}
	\label{eq:extensions}
	\phi_j   \colon B_j' \setminus L_{j,q_j} \rightarrow B_j^{\rm orb} \setminus \{s_j\} 
\end{equation}
of the map~$\phi_0\vert_{\Sigma'_j}$, where~$L_{j,q_j}$ are Lagrangian pinwheels for~$j \in \{a,b,c\}$. For this we use Lemma~\ref{lem:local}. Again, restricting our attention to~$a$, let~$\psi_a$ be the symplectomorphism from Lemma~\ref{lem:local}. Note that we have already established the correspondence between~$B^{\rm orb}_a$ and~$\C^2/\Gamma_{a,q}$ by~$\rho_a$ and that this correspondence is compatible with the toric picture. We now establish a correspondence between the rational homology sphere~$B'_{a}$ and the space~$S_{a,q}$ coming from the smoothing in Lemma~\ref{lem:local}. Define yet another copy~$\Sigma_{a}''$ of the lens space by setting~$\Sigma''_{a} = \psi_a^{-1}(\rho_a(\Sigma_{a,q}))$. This lens space is also equipped with the standard contact structure and it bounds a rational homology ball~$B_a''$ by~\cite[Example 2.5]{EvaSmi18}. We now have two pairs~$(B_a',\Sigma'_a)$ and~$(B_a'',\Sigma''_a)$ consisting of a rational homology ball bounded by a lens space carrying its standard contact structure. By~\cite[Proposition A.2]{GolSta21} which relies on~\cite{Lis08}, this implies that~$(B_a',\Sigma'_a)$ and~$(B_a'',\Sigma''_a)$ are equivalent up to symplectic deformation. Let~$\chi_a \colon (B_a',\Sigma_a') \rightarrow (B_a'', \Sigma_a'')$ be the diffeomorphism we obtain from this. Note that we cannot directly use the symplectic deformation to conclude, since the symplectomorphism obtained from a Moser-type argument may not restrict to the desired map on the boundary. More precisely, we obtain a diagram of diffeomorphisms of lens spaces,
\begin{equation}
	\label{eq:diagram}
	\begin{tikzcd}
		B_a' \supset \Sigma'_a \arrow{r}{\phi_0\vert_{\Sigma'_a}} \arrow{d}{\chi_a\vert_{\Sigma'_a}}
		& \Sigma_a \subset B_a^{\rm orb} \arrow{d}{\rho_a\vert_{\Sigma_a}} \\
		B_a'' \supset \Sigma''_a \arrow{r}{\psi_a\vert_{\Sigma_a''}}
		& S^3(d)/\Gamma_{a,q} \subset B^4(d)/\Gamma_{a,q},
	\end{tikzcd}
\end{equation} 
and this diagram \emph{does not commute}. We may however correct the diffeomorphism~$\chi_a$ so that~\eqref{eq:diagram} commutes. Recall that~$\Sigma_{a,q}$ is a lens space of type~$(aq-1,a^2)$. Since~$(aq-1)^2 \neq \pm 1 \mod a^2$, it follows from~\cite[Th\'eor\`eme 3 (a)]{Bon83} that the space of diffeomorphisms of~$\Sigma_a''$ has two components, namely the one of the identity and the one of the involution~$\tau$ induced by the involution $(z_1,z_2) \mapsto (\overline z_1, \overline z_2)$ of~$S^3$. The diffeomorphism~$\tau$ extends to a diffeomorphism~$\tilde{\tau}$ of~$B''_{p,q}$. Up to post-composing~$\chi_a$ with~$\tilde{\tau}$, we may thus assume that~$(\psi_a^{-1}\rho_a\phi_0\chi_a^{-1})\vert_{\Sigma_a''}$ is isotopic to the identity by an isotopy~$\varphi_t$. Using this isotopy, we can correct the diffeomorphism~$\chi_a$ such that the diagram~\eqref{eq:diagram} commutes. Indeed, the set~$\pi^{-1}(V_a \cap W)$ is a collar neighbourhood~$\Sigma'_a \times [0,2)$ and thus we can use the collar coordinate together with the isotopy~$\varphi_t$ to define a corrected diffeomorphism~$\tilde{\chi}_a$ which coincides with the original diffeomorphism~$\chi_a$ on~$\mu^{-1}(V_a \setminus W)$ and with~$(\psi_a^{-1}\rho_a\phi_0)\vert_{\Sigma_a'}$ on~$\Sigma'_a$. Recall that two collar neighbourhoods which agree on the boundary coincide up to applying a smooth isotopy, see Munkres~\cite[Lemma 6.1]{Mun16}. This means that, after applying an isotopy in~$(B_a'',\Sigma_a'')$, we can assume that the corrected version of~$\chi_a$ and~$\psi_a^{-1}\rho_a\phi_0$ agree on a smaller collar~$\Sigma_a \times [0,1)$. Denoting the diffeomorphism we obtain in this way by~$\tilde{\chi}_a$, this allows us to define a diffeomorphism
\begin{equation}
	\label{eq:phij}
	\phi_a = \rho_a^{-1} \circ \psi_a \circ \tilde{\chi}_a \vert_{B_a' \setminus L_{a,q_a}} \colon
	B_a' \setminus L_{a,q_a} \rightarrow B^{\rm orb}_a \setminus \{s_a\},
\end{equation}
which extends~$\phi_0$ in the sense that it agrees with~$\phi_0$ on a collar of~$\Sigma_a'$. Since~$\psi_a$ is defined outside of a Lagrangian pinwheel~$L_{a,q_a} \subset B_a''$, the diffeomorphism~$\phi_a$ is defined outside of a pinwheel (which we again denote by~$L_{a,q_a}$) in~$B_a'$. We repeat this procedure for~$b$ and~$c$ to obtain diffeomorphisms~$\phi_b$ and~$\phi_c$.

\emph{Step 4: } By construction, the diffeomorphisms~$\phi_a,\phi_b,\phi_c$ extend the initial symplectomorphism~$\phi_0 \vert_{\pi^{-1}(W \setminus (V_a \cup V_b \cup V_c) )}$ and hence we obtain a diffeomorphism
\begin{equation}
	\label{eq:phihat}
	\widehat{\phi} \colon 
	\CP^2 \setminus (L_{a,q_a} \sqcup L_{b,q_b} \sqcup L_{c,q_c})
	\rightarrow 
	\CP(a^2,b^2,c^2) \setminus \{s_a,s_b,s_c\} .
\end{equation}
We now turn to the symplectic forms. On~$\CP^2$, we define a symplectic form~$\widehat{\omega}$ which turns~$\widehat{\phi}$ into a symplectomorphism as follows. On~$\pi^{-1}(W \setminus (V_a \cup V_b \cup V_c) )$ we define~$\widehat{\omega}$ to be the usual Fubini-Study form~$\omega$. On~$V_j$ we define~$\widehat{\omega}$ as the pullback form~$\tilde{\chi_j}^* \omega_{B_j''}$, where~$\tilde{\chi}_j$ is the corrected diffeomorphism constructed at the end of \emph{Step 3}. This yields a well-defined symplectic form which turns~$\widehat{\phi}$ into a symplectomorphism. Indeed, this follows from the fact that the maps~$\phi_0$,~$\psi_j$ and~$\rho_j$ are symplectomorphisms and~$\phi_j$ is defined as their composition~\eqref{eq:phij}. This also implies that the symplectic form~$\widehat{\omega}$ has the same total volume as the Fubini--Study form. By the Gromov--Taubes theorem~\cite[Remark 9.4.3 (ii)]{MS12}, the form~$\tilde{\omega}$ is symplectomorphic to the the Fubini--Study form and hence post-composing~$\widehat{\phi}$ from~\eqref{eq:phihat} with this symplectomorphism yields the desired symplectomorphism~\eqref{eq:mainsymplecto}.

\emph{Step 5: } The definition of the global map~$\tilde{\phi} \colon \CP^2 \rightarrow \CP(a^2,b^2,c^2)$ is obtained by replacing~$\phi_a$ from~\eqref{eq:phij} by~$\tilde{\phi}_a = \rho_a^{-1} \circ \tilde{\psi}_a \circ \tilde{\chi}_a$ and carrying out the rest of the construction as above. The map~$\tilde{\psi}_a$ is given by Lemma~\ref{lem:local}. Let us now identify~$\tilde{\phi}^{-1}(\cd)$, where~$\cd = \mu^{-1}(\pp \Delta_{a,b,c})$. Let~$\tilde{W} \subset W$ be the subset of~$W$ where~$\tilde{\phi}$ coincides with~$\phi_0$. Since~$\phi_0$ intertwines the toric structures on~$\pi^{-1}(W) \subset \CP^2$ and~$\mu^{-1}(W) \subset \CP(a^2,b^2,c^2)$, the set~$\tilde{\phi}^{-1}(\cd) \cap \pi^{-1}(\tilde{W})$ fibers over the three pieces of the boundary given by~$\tilde{W} \cap \pp\Delta_{a,b,c}$ and hence consists of three disjoint symplectic cylinders. 

We use Lemma~\ref{lem:local} to prove that the missing pieces~$\tilde{\phi}^{-1}(\cd) \cap B_j'$ for~$j \in \{a,b,c\}$ are also given by symplectic cylinders and that the union of the six cylinders is given by a torus. We again discuss the case of~$j = a$ since the other two are completely analogous. Recall that~$\rho_a \colon B_a^{\rm orb} \rightarrow B^4(d)/\Gamma_{a,q}$ is compatible with the toric structure and thus~$\rho_a(\cd \cap B_a^{\rm orb}) = \cd_{a,q_a} \cap \rho_a(B_a^{\rm orb})$, where~$\cd_{a,q_a} = \{ z_1z_2 = 0 \}/\Gamma_{a,q_a}$ as in Lemma~\ref{lem:local}. Indeed, this follows from the fact that~$\cd_{a,q_a}$ fibers over the boundary of the moment map image~$\angle_{a,q_a}$. By Lemma~\ref{lem:local}, we deduce that~$\tilde{\chi}_a^{-1}(\tilde{\phi}_a^{-1}(\rho_a(\cd \cap B_a^{\rm orb})))$ is given by the union of a pinwheel with a piece of a cylinder. Furthermore, recall that near the lens spaces at the respective boundaries, the map~$\tilde{\rho}_a^{-1}\tilde{\psi}_a\chi_a$ coincides with~$\phi_0$. This implies that the cylinder contained in~$B_a'$ has two boundary components at~$\pp B_a' = \Sigma_a'$ which are smoothly identified in a collar neighbourhood with boundary components of the set~$\tilde{\phi}^{-1}(\cd) \cap \pi^{-1}(\tilde{W})$ discussed in the previous paragraph. This proves the claim. 
\proofend

\begin{remark}
	{\rm 
		We suspect that there are shorter and more natural proofs of Proposition~\ref{lem:main}. In particular, one should be able to avoid Gromov--Taubes. One possibility we have hinted at above is working with a global degeneration and trying to analyze its vanishing cycle. This would completely avoid the use almost-toric fibrations. Another possibility, in the spirit of~\cite{Rua01} and~\cite{HarKav15} is to equip the explicit local degeneration from~\cite{LekMay14} with a family of integrable systems avoiding the pinwheel and extending the given toric structure on the boundary. The symplectomorphism from Proposition~\ref{lem:main} then follows from the usual toric arguments and it is automatically equivariant. Although this construction is elementary, it is somewhat outside the scope of this appendix and we hope to carry out the details elsewhere. This is also reminiscent of~\cite[Section 7]{GroVar21} and it is plausible one can apply results from this paper to prove the same result.
	}
\end{remark}

\bibliographystyle{amsalpha}
\bibliography{bibliographyssvz, mybibfile2}

\end{document}